\documentclass[11pt]{article}
\usepackage{geometry}                
\usepackage[colorinlistoftodos]{todonotes}
\usepackage{tikz-cd}
\usepackage{graphicx}
\usepackage{amssymb}
\usepackage{epstopdf}
\usepackage[T1]{fontenc}
\usepackage[ngerman, english]{babel}
\usepackage{mathtools}
\usepackage{bbm}
\usepackage{mathrsfs}
\usepackage{amsmath,amscd,amsthm}
\usepackage{tikz}
\usepackage{tikz-cd}
\usepackage{babel}
\usetikzlibrary{arrows, matrix}
\usetikzlibrary{cd}
\usepackage[arrow, matrix, curve]{xy}

\newtheorem{theorem}{Theorem}[section]
\newtheorem{lemma}[theorem]{Lemma}
\newtheorem{proposition}[theorem]{Proposition}
\newtheorem{corollary}[theorem]{Corollary}
\newtheorem{conjecture}[theorem]{Conjecture}
\newtheorem{definition}[theorem]{Definition}

\DeclareMathOperator{\Gal}{\operatorname{Gal}}
\DeclareMathOperator{\Q}{\mathbf{Q}}

\DeclareMathOperator{\Z}{\mathbf{Z}}
\DeclareMathOperator{\F}{\mathbf{F}} 
\DeclareMathOperator{\G}{\mathscr{G}}

\DeclareMathOperator{\N}{\mathbf{N}}

\DeclareMathOperator{\Ker}{\mathrm{ker}}

\DeclareMathOperator{\Spec}{\operatorname{Spec}}
\DeclareMathOperator{\Spf}{\operatorname{Spf}}

\DeclareMathOperator{\Lie}{\mathrm{Lie}}

\DeclareMathOperator{\Og}{\mathcal{O}}

\DeclareMathOperator{\Pic}{\mathrm{Pic}}

\DeclareMathOperator{\rk}{\mathrm{rk}}

\DeclareMathOperator{\Gr}{\mathrm{Gr}}

\DeclareMathOperator{\Gm}{\mathbf{G}_m}

\DeclareMathOperator{\fppf}{\mathrm{fppf}}

\DeclareMathOperator{\Res}{\mathrm{Res}}

\DeclareMathOperator{\sep}{{^\mathrm{sep}}}

\usepackage{xcolor}
\usepackage[only,mapsfrom]{stmaryrd}
\theoremstyle{remark}
\newtheorem{remark}[theorem]{Remark}
\title{Chai's conjectures on base change conductors}
\date{}
\author{Otto Overkamp and Takashi Suzuki}
\begin{document}
\maketitle

{\abstract{The base change conductor is an invariant introduced by Chai which measures the failure of a semiabelian variety to have semiabelian reduction. We investigate the behaviour of this invariant in short exact sequences, as well as under duality and isogeny. Our results imply Chai's conjecture on the additivity of the base change conductor in short exact sequences, while also showing that a proposed generalisation of this conjecture fails. We use similar methods to show that the base change conductor is invariant under duality of Abelian varieties in equal positive characteristic (answering a question of Chai), as well as giving a new short proof of a formula due to Chai, Yu, and de Shalit which expresses the base change conductor of a torus in terms of its (rational) cocharacter module.}}
\tableofcontents
\section{Introduction} Let $\Og_K$ be a Henselian discrete valuation ring with field of fractions $K$ and \it perfect \rm residue field $\kappa.$ Let $B$ be a semiabelian variety over $K.$ It is well-known that $B$ admits a Néron lft-model $\mathscr{B} \to \Spec \Og_K,$ which is a smooth separated $\Og_K$-group scheme with generic fibre $B$ characterised by a universal property. We say that $B$ has \it semiabelian reduction \rm over $\Og_K$ if the identity component $\mathscr{B}^0$ of the Néron lft-model of $B$ is a semiabelian scheme. If $L$ is any finite separable extension of $K,$ we let $\Og_L$ be the integral closure of $\Og_K$ in $L.$ Note that $\Og_L$ is a discrete valuation ring finite and flat over $\Og_K$. By Grothendieck's semiabelian reduction theorem, there exists a finite Galois extension $L$ of $K$ such that $B_L$ has semiabelian reduction over $\Og_L.$ Choose such an $L$ and let $\mathscr{B}_L$ be the Néron lft-model of $B_L$ over $\Og_L.$ Following Chai \cite{Chai}, we define the \it base change conductor \rm of $B$ by
$$c(B):= \frac{1}{e_{L/K}} \ell_{\Og_L} (\mathrm{coker}(\Lie \mathscr{B} \otimes_{\Og_K} \Og_L \to \Lie \mathscr{B}_L)),$$
where $e_{L/K}$ denotes the ramification index of the extension $K\subseteq L$ and $\ell_{\Og_L}(-)$ denotes the length of an $\mathcal{O}_{L}$-module. This rational number measures the failure of $B$ to have semiabelian reduction over $\Og_K$ (or, equivalently, the failure of the connected component of the Néron lft-model of $B$ to commute with base change along the possibly ramified extension $\Og_K \subseteq \Og_L$). Moreover, $c(B)$ is independent of the choice of $L$. 

Now let 
\begin{align}0\to T\to B \to A \to 0\label{introsec}\end{align}
be an exact sequence of semiabelian varieties over $K.$ Because the induced sequence $0\to \mathscr{T} \to \mathscr{B} \to \mathscr{A} \to 0$ of Néron lft-models is usually not exact, understanding the behaviour of the base change conductor in short exact sequences is a highly delicate problem. The following two questions have been proposed in the literature (see, e. g., \cite{Chai}, \cite{CLN}): 
\begin{itemize} 
\item[(i)] Suppose $T$ is a torus and $A$ an Abelian variety in the sequence (\ref{introsec}) above. Does the equality $c(B)=c(T) + c(A)$ hold?
\item[(ii)] Does the equality $c(B)=c(T)+c(A)$ hold without any restrictions on $T,$ $B,$ and $A$?
\end{itemize}
Question (i) above was answered affirmatively by Chai in the special cases where $K$ has characteristic zero or where $\kappa$ is finite \cite[Theorem 4.1]{Chai}; the first of those two cases was later re-proven by Cluckers-Loeser-Nicaise \cite[Theorem 4.2.4]{CLN} using different methods. This question has since been widely believed to have an affirmative answer in general, which has become known as \it Chai's conjecture. \rm Beyond the results just described, this conjecture is known in the following two cases: 
\begin{itemize}
\item The semiabelian variety $B$ acquires semiabelian reduction over a finite \it tamely ramified \rm extension of $K,$ due to Halle-Nicaise \cite[Corollary 4.23]{HN}, and 
\item we have $B=\Pic ^ 0_{C/K}$ for a proper curve $C\to \Spec K,$ due to the first named author \cite{OvII}. 
\end{itemize}
The first main purpose of the present paper is to settle both questions (i) and (ii) in complete generality, without imposing any restrictions on the characteristic of $K$ or the semiabelian variety $B$. More precisely, we shall show that a slightly stronger statement than Chai's conjecture is true, whereas question (ii) has a negative answer: 
\begin{theorem} \label{intromainthm1}
(i) Assume that $T$ is a torus in the sequence (\ref{introsec}) above. Then $$c(B)=c(T)+c(A).$$
(ii) If $T$ is not a torus in the sequence (\ref{introsec}), then $c(B) \not= c(T) + c(A)$ in general. In fact, there are such examples where $T,$ $B,$ and $A$ are all Abelian varieties, and examples where $T$ is an Abelian variety and $A$ a torus.
\end{theorem}

See Theorem \ref{Chaithm} for Statement (i) and Subsections \ref{consIsubs} and \ref{consIIsubs} for Statement (ii).
Our results have applications in several (seemingly unrelated) fields of mathematics: 
\begin{itemize}
\item It is well-known that every Abelian variety $Y$ over $K$ has a \it rigid uniformisation \rm $0\to \Lambda \to B \to Y \to 0,$ where $\Lambda$ is an étale $K$-lattice, $B$ a semiabelian variety, and the sequence is to be understood in the category of rigid analytic $K$-groups (see, e. g., Section 1 of \cite{BX} for more details). Moreover, $B$ fits into an exact sequence $0\to T \to B \to A \to 0$ where $T$ is a torus and $A$ an Abelian variety with potentially good reduction. By \cite[Theorem 2.3]{BX}, we have $c(Y)=c(B),$ so by our results, calculating the base change conductor of $Y$ is reduced to calculating that of $T$ and that of $A.$ 

\item

If $K$ is the function field of a proper smooth curve $X$
over a finite field $\kappa = \F_{q}$ (for this paragraph only),
then by \cite{GS}, the coherent Euler characteristic
$\chi(X, \Lie \mathscr{B})$
appears in a special value formula for the $L$-function of $B$.
In the proof of \cite[Proposition 7.6]{GS},
the affirmative answer for Question (i) for finite $\kappa$,
in the form of the equality
	\[
			\chi(X, \Lie \mathscr{B})
		=
				\chi(X, \Lie \mathscr{T})
			+
				\chi(X, \Lie \mathscr{A}),
	\]
is used to reduce the special value formula to
those of the torus part $T$ and the abelian variety quotient $A$.
It is easy to see that $L$-functions are multiplicative
in exact sequences of semiabelian varieties.
Therefore one might expect this kind of additivity of $\chi(X, \Lie \mathrm{N\acute{e}ron})$
for any exact sequence of semiabelian varieties over $K$,
so that $\chi(X, \Lie \mathscr{B}^{\bullet})$ would be an invariant
of any object $B^{\bullet}$
of the bounded derived category of semiabelian varieties over $K$
(or of mixed motives over $K$ even more ambitiously).
However, our counterexamples to Question (ii) show that
this is not the case:
$\chi(X, \Lie \mathscr{B}^{\bullet})$ is not well-defined
(namely, not invariant under quasi-isomorphism).
What is well-defined instead
(under the assumption of the finiteness of Tate-Shafarevich groups) is
the \emph{ratio} of $q^{\chi(X, \Lie \mathscr{B}^{\bullet})}$
to the Weil-\'etale Euler characteristic
$\chi_{W}(X, \mathscr{B}^{0, \bullet})$.
This will be discussed in detail elsewhere.

\item As explained in Theorem 4.2.1 of  \cite{CLN}, Chai's conjecture is equivalent to a Fubini property of certain motivic integrals, which was hitherto known only in the case where $K$ has characteristic 0. Our results show that this Fubini property holds without any restriction on the characteristic of $K.$
\end{itemize}
Finally, we shall extend Chai's result on the invariance of the base change conductor under duality of Abelian varieties to the equal characteristic case:
\begin{theorem} \label{DualInv}
Let $A$ be an Abelian variety over $K$ with dual $A^\vee.$ Then $c(A)=c(A^\vee).$ \label{intromainthm2}
\end{theorem}
This is non-trivial since base change conductors for Abelian varieties are not isogeny invariant as shown in \cite[(6.10)]{Chai}. The Theorem is known if the characteristic of $K$ is equal to 0 \cite[Theorem 6.7]{Chai}, but the proof given in \it loc. cit. \rm cannot easily be generalised to the positive characteristic case due to the existence of inseparable isogenies.
In \cite{Suz25},
this result is applied to prove the duality invariance of Faltings heights
of abelian varieties over global function fields.

The same method of proof, when applied to tori instead,
gives a new proof of the following result of
Chai, Yu, and de Shalit \cite[Theorem on p. 367 and Theorem 12.1]{CY}:

\begin{theorem}
    Let $T$ and $T'$ be tori over $K$ isogenous to each other.
    Then $c(T) = c(T')$. \label{intromainthm3}
\end{theorem}

Chai, Yu, and de Shalit prove this first for mixed characteristic $K$
using B\'egueri's duality theory \cite{Beg}.
They then deduce the equal characteristic case
from the mixed characteristic case
by a Deligne-type approximation argument.
Here we directly prove the equal characteristic case of the above theorem
without reduction to the mixed characteristic case.
As mentioned in \cite{CY}, this theorem implies that
	\[
		c(T) = \frac{1}{2} a(X_{\ast}(T) \otimes_{\Z} \Q)
	\]
for any torus $T / K$,
where $a(-)$ denotes the Artin conductor.

We shall also show that the assumption in Chai's conjecture that $\kappa$ be perfect cannot be dropped, and that the base change conductor of algebraic tori is no longer invariant under isogeny if the residue field is imperfect. In particular, the formula just stated does not extend to the case of imperfect residue fields. 

The methods we shall employ come from several different sources. Our approach for proving Chai's conjecture is, to some extent, inspired by the first named author's previous work on this topic \cite{OvII}. The proof of Chai's conjecture for Jacobians given there relies both on the methods presented in the work of Liu-Lorenzini-Raynaud \cite{LLR}, as well as on the techniques for studying Néron models of Jacobians of singular curves introduced by the first named author \cite{OvI,Ov}. We shall show in this article that the use of such techniques can, in fact, be circumvented, thus leading to a stronger result. The key results from \cite{LLR} we use here are the formulas \cite[Theorems 2.1 (b) and 2.10]{LLR} writing lengths of cohomology objects of complexes of Lie algebras as dimensions of certain algebraic groups over the residue field $\kappa$.%
\footnote{
    Note that some errors in \cite{LLR} have been corrected
    in the corrigendum \cite{LLRII}.
    We only use results of \cite{LLR} unaffected by these errors.
}

The methods we shall use in order to construct the counterexamples mentioned in Theorem \ref{intromainthm1} (ii), as well as to prove Theorems \ref{intromainthm2} and \ref{intromainthm3}, are largely derived from the version of arithmetic duality theory developed by the second named author \cite{Suz}. This theory treats dualities for algebraic groups over $\kappa$ arising from those over $K$, and we will apply it to Liu-Lorenzini-Raynaud's algebraic groups mentioned above.

Another ingredient for counterexamples is
Tan-Trihan-Tsoi's estimates \cite{TTT} of
the kernel of the restriction map
$H^{1}(K, A) \to H^{1}(L, A)$
for ordinary elliptic curves $A$ with supersingular reduction
and highly ramified extensions $L / K$.

After a preprint version of this paper was uploaded to arXiv, the authors learned that Vologodsky also obtained the finite residue field case of Theorem \ref{DualInv} independently at more or less the same time.%
\newsavebox\myVerb%
\begin{lrbox}\myVerb%
    \footnotesize\verb|https://researchseminars.org/talk/MITNT/71/|%
\end{lrbox}%
\footnote{
    The abstract and the video of Vologodsky's talk at MIT
    on March 8th, 2023 titled
    ``Dual abelian varieties over a local field have equal volumes''
    are available at:
    \par\noindent
    \usebox\myVerb
}
His proof uses Tate-Milne's duality (for finite residue fields) while our proof uses the second named author's duality \cite{Suz} (for general perfect residue fields) among other similarities and differences.

Throughout this paper (except for Subsection \ref{ImperfectResPara}), we shall assume that $\Og_K$ is complete and that its residue field (henceforth to be denoted by $k$) is algebraically closed. This leads to no loss of generality because Néron lft-models commute with base change along the extension $\Og_K \subseteq \widehat{\Og_K^{\mathrm{sh}}}$ \cite[Chapter 10.1, Proposition 3]{BLR}. \\
\\
$\mathbf{Acknowledgement.}$ The authors would like to express their gratitude to Ki-Seng Tan for suggesting that his joint paper \cite{TTT} with Trihan and Tsoi might be relevant for Construction I (see Proposition \ref{ConIProp}) below. Moreover, they are grateful to Vadim Vologodsky for the information about his work and the anonymous referee for a large number of very valuable comments. The first named author's research was   conducted in the framework of the research training group \emph{GRK 2240: Algebro-Geometric Methods in Algebra, Arithmetic and Topology}.
\section{Chai's conjecture}
Throughout this section, we let $\Og_K$ be a complete discrete valuation ring with field of fractions $K,$ maximal ideal $\mathfrak{m},$ and algebraically closed residue field $k=\Og_K/\mathfrak{m}.$ 
\subsection{Some cohomological invariants}

Let 
$$A^\bullet \colon \, 0 \to A^1 \to A^2 \to ... \to A^n \to 0 $$ 
be a bounded complex of finitely generated free $\Og_K$-modules such that the induced complex $0\to A^1\otimes_{\Og_K}K \to ... \to A^n\otimes_{\Og_K}K\to 0$ is exact (i. e., $A^\bullet{\otimes_{\Og_K}^L}K\cong 0$ in $D^{b}(K)$). As usual, we shall write 
$$\det A^{\bullet} := \bigotimes_{i=1} ^n  (\det A^i)^{\otimes (-1)^i},$$ where $\det A^j := \bigwedge^{\rk_{\Og_K} A^j} A^ j$ for all $j.$ Then we obtain a morphism of $\Og_K$-modules 
$$\det A^\bullet \to \det A^\bullet{\otimes_{\Og_K}^L}K = K,$$ so there exists a unique integer $\gamma(A^\bullet)$ such that 
$\det A^\bullet = \mathfrak{m}^{-\gamma(A^\bullet)}$ as an $\Og_K$-submodule of $K.$ Moreover, we define the \it cohomological Euler characteristic \rm of $A^\bullet$ as
$$\chi(A^\bullet):=\sum_{i=1}^n (-1)^ i\ell_{\Og_K}(H^i(A ^ \bullet)).$$ Then we have
\begin{lemma}
For each complex $A^\bullet$ as above, $\gamma(A^\bullet)=\chi(A^\bullet).$ \label{chigammalem}
\end{lemma}
\begin{proof}
If $n=1,$ both invariants vanish, so there is nothing to prove. If $n=2,$ we choose a basis $e_1,..., e_r$ of $A^2$ such that $\lambda_1e_1,..., \lambda_r e_r$ is a basis of $A^1$ for some $r\in \N_0$ and suitable $\lambda_j\in \Og_K.$ Since $A^1\to A^2$ is injective, we have $\chi(A^\bullet)=v_K(\lambda_1) +...+ v_K(\lambda_r).$ Moreover, $\det A ^ \bullet$ is generated by $(\lambda_1\cdot...\cdot \lambda_r)^{-1}\cdot(e_1\wedge ... \wedge e_r)^\vee\otimes (e_1\wedge ... \wedge e_r)$ as a submodule of $\det A^\bullet{\otimes_{\Og_K}^L}K=K,$ which implies the claim.

In general, consider the two complexes $B^\bullet: \, 0\to A^1 \to ... \to A^{n-2} \to \ker(A^{n-1}\to A^n) \to 0$ and $C^\bullet: \, 0 \to \mathrm{im}(A^{n-1}\to A^n) \to A^n \to 0,$ declaring that $\mathrm{im}(A^{n-1}\to A^n)$ sit in degree $n-1.$ Then we have a (term-wise) exact sequence of complexes
$$0 \to B^\bullet \to A^\bullet \to C^\bullet \to 0,$$
which induces a canonical isomorphism 
$$\det A^\bullet = (\det B^\bullet)\otimes_{\Og_K} \det C^\bullet.$$ This isomorphism shows that $\gamma(A^\bullet)=\gamma(B^\bullet) + \gamma(C^\bullet),$ and the long exact cohomology sequence induced by the short exact sequence of complexes shows that $\chi(A^\bullet)=\chi(B^\bullet) + \chi(C^\bullet).$  Hence the claim follows in general by induction.
\end{proof}

\subsection{Complexes of group schemes}
In this paragraph, we shall recall the construction of another invariant which will play an important role in this article. Our discussion is based on \cite[Section 2]{LLR}. Let 
$$\mathscr{G}^\bullet: \, 0 \to \mathscr{G}^1 \to \mathscr{G}^2 \to... \to \mathscr{G}^{n-1} \to \mathscr{G}^{n} \to 0$$ be a complex of separated group schemes of finite type over $\Og_K.$ Moreover, putting $G^i:=\mathscr{G}^i\times_{\Og_K} K$, we shall assume that the induced complex
$$G^\bullet:\, 0\to G^1\to G^2 \to ... \to G^{n-1} \to G^n \to 0$$ is exact and consists of smooth algebraic $K$-groups. Denote by $\mathscr{G}^\bullet(\Og_K)$ the complex
$0 \to \mathscr{G}^1(\Og_K) \to \mathscr{G}^2(\Og_K) \to ... \to \mathscr{G}^n(\Og_K)\to 0.$ We shall need the following
\begin{definition}
Let $\Lie \mathscr{G}^\bullet$ be the complex $0\to \Lie\mathscr{G}^1 \to \Lie\mathscr{G}^2\to ...  \to\Lie\mathscr{G}^n\to 0.$ Moreover, we put
$$\gamma(\mathscr{G}^\bullet):=\gamma(\Lie\mathscr{G}^\bullet),$$ and $$\chi_{\mathrm{Lie}}(\mathscr{G}^\bullet):=\chi(\Lie\mathscr{G}^\bullet).$$ 
\end{definition}
Now recall (for example, from \cite[Chapter 7.1, Theorem 5]{BLR}) that any $\Og_K$-group scheme $\mathscr{K}$ of finite type with smooth generic fibre admits a \it smoothening \rm $$\psi\colon \widetilde{\mathscr{K}}\to \mathscr{K},$$
i, e., a smooth group scheme $\widetilde{\mathscr{K}}$ and a morphism $\psi$ as above such that, for any smooth $\Og_K$-scheme $T,$ any morphism $T\to \mathscr{K}$ factors uniquely via $\psi.$ This universal property guarantees in particular that the smoothening is functorial in $\mathscr{K},$ so that, given a complex $\mathscr{G}^\bullet$ as above, we can consider the induced complex
$$\widetilde{\mathscr{G}}^\bullet : \, 0 \to \widetilde{\mathscr{G}}^1 \to \widetilde{\mathscr{G}}^2 \to ... \to \widetilde{\mathscr{G}}^n\to 0.$$
The following Lemma is of central importance for the construction (and used implicitly throughout Section 2.5 of \cite{LLR}):
\begin{lemma} \label{commuteswithcokerlemma}
Let $\kappa$ be an \rm algebraically closed \it field and let $f\colon G\to H$ be a morphism of smooth algebraic groups over $\kappa.$ Then $(\mathrm{coker}\,f)(\kappa)=\mathrm{coker}\, f(\kappa),$ where $f(\kappa)\colon G(\kappa)\to H(\kappa)$ is the induced map on $\kappa$-points.
\end{lemma}
\begin{proof}
Let $H'$ be the scheme-theoretic image of $f,$ which is a closed algebraic subgroup of $H$ smooth over $\kappa.$ The sequence $0\to H'\to H\to \mathrm{coker}\, f\to 0$ is exact in the étale topology, so $(\mathrm{coker}\,f)(\kappa)=H(\kappa)/H'(\kappa).$ Now let $g\colon G\to H'$ be the induced morphism. For every $x\in H'(\kappa),$ the scheme $g^{-1}(x)$ is non-empty and of finite type over $\kappa,$ so it has a $\kappa$-point because $\kappa$ is algebraically closed. This shows that $H'(\kappa)=\mathrm{im} \, f(\kappa),$ so the claim follows. 
\end{proof}

The Lemma fails if $\kappa$ is separably closed but imperfect, as shown by $f=[p]\colon \Gm\to \Gm$ with $p=\mathrm{char}\, k.$ Indeed, $\mathrm{coker}\, f=0$ in this case, but $\mathrm{coker}\, f(\kappa)=\kappa^\times/\kappa^{\times p},$ which is infinite. 
\begin{proposition}\rm (cf. \cite[p. 473]{LLR}) \it 
For $i=1,..., n,$ there exist canonical smooth group schemes $D_i$ of finite type over $k$ and isomorphisms
$$H^  i (\mathscr{G}^\bullet(\Og_K)) \cong D_i(k)$$ of Abelian groups. \label{Diconstructprop}
\end{proposition}
\begin{proof}
We briefly recall the construction of the $D_i,$ referring the reader to \it loc. cit. \rm for more details. Let $\mathscr{K}$ and $\mathscr{K}'$ be the smoothenings of $\mathscr{G}^ {i-1}$ and $\ker(\mathscr{G}^ i \to \mathscr{G}^{i+1}),$ respectively. We obtain a morphism $\mathscr{K}\to \mathscr{K}'$ and a canonical isomorphism
$$H^ i (\mathscr{G}^\bullet(\Og_K)) = \mathrm{coker}(\mathscr{K}(\Og_K) \to \mathscr{K}'(\Og_K)).$$ As in \cite[p. 465]{LLR}, there exists a canonical diagram
$$\begin{CD}
\widehat{\mathscr{K}}@>>>\widehat{\mathscr{K}'}\\
@VVV@VVV\\
\mathscr{K}@>>>\mathscr{K}'
\end{CD}$$
of smooth $\Og_K$-group schemes of finite type such that the vertical arrows are isomorphisms at the generic fibre and such that the map $\widehat{\mathscr{K}}\to \widehat{\mathscr{K}'}$ is smooth and surjective. Denote by $\Gr_j$ the Greenberg functor of level $j,$ for any $j\geq 0$ (see, for example, \cite[ Section 7]{BGA}). We obtain a diagram of smooth $k$-group schemes of finite type
$$\begin{CD}
\Gr_j \widehat{\mathscr{K}}@>>>\Gr_j\widehat{\mathscr{K}'} @>>>0\\
@VVV@VVV\\
\Gr_j\mathscr{K}@>>>\Gr_j\mathscr{K}'\\
@VVV@VVV\\
C@>>> C''
\end{CD}$$
in which $C$ and $C''$ are the obvious cokernels. The top left horizontal arrow is smooth and surjective (and hence surjective on $k$-points). By  \cite[Theorem 2.2(c) and Corollary 2.3]{LLR}, as well as Lemma \ref{commuteswithcokerlemma} above, there exists $N\in \N$ such that for all $j\geq N,$ the maps $\mathrm{coker}(\widehat{\mathscr{K}}(\Og_K) \to \mathscr{K}(\Og_K)) \to C(k)$ and $\mathrm{coker}(\widehat{\mathscr{K}'}(\Og_K) \to \mathscr{K}'(\Og_K)) \to C''(k)$ are isomorphisms. Now put 
$$D_i:=\mathrm{coker}(C\to C'').$$ Then the diagrams above show that the canonical map
$$H^ i (\mathscr{G}^\bullet(\Og_K)) \to D_i(k)$$ is an isomorphism.
\end{proof}

\begin{remark} \label{kersmooth}
	The groups $C$ and $C''$ in the proof of Proposition \ref{Diconstructprop}
	are independent of $j \ge N$.
	More precisely, with these groups denoted by $C_{j}$ and $C_{j}''$,
	the natural morphisms
	$C_{j} \to C_{j - 1}$ and $C_{j}'' \to C_{j - 1}''$
	(denoted by
	$\operatorname{Coker}(g_{j}) \to \operatorname{Coker}(g_{j - 1})$
	and
	$\operatorname{Coker}(g_{j}'') \to \operatorname{Coker}(g_{j - 1}'')$
	in \cite[p. 471]{LLR})
	are isomorphisms.
	The proof of this fact in \cite{LLR} refers to Lemma 2.6 (c) there,
	which says that these morphisms are isomorphisms
        \emph{on $k$-valued points}:
	$C_{j}(k) \overset{\sim}{\to} C_{j - 1}(k)$ and
	$C_{j}''(k) \overset{\sim}{\to} C_{j - 1}''(k)$.
	A priori, this only implies that
	$C_{j} \to C_{j - 1}$ and $C_{j}'' \to C_{j - 1}''$
	are surjective with finite infinitesimal kernel.
	
	To show that they are indeed isomorphisms,
	we use the following fact given in \cite[Proposition 11.1]{BGA}:
	If $\mathscr{G}$ is a smooth group scheme over $\mathcal{O}_{K}$,
	then the natural morphism
	$\operatorname{Gr}_{j} \mathscr{G} \to \operatorname{Gr}_{j - 1} \mathscr{G}$
	is smooth for all $j$.
	Apply this to $\mathscr{G} = \mathscr{K}$.
	Together with the smoothness of
	$\operatorname{Gr}_{j - 1} \widehat{\mathscr{K}} $ over $k$,
	we know that the kernel of $C_{j} \to C_{j - 1}$ is smooth over $k$.
	Hence this kernel is zero for $j \ge N$.
	The same argument applies to $C_{j}'' \to C_{j - 1}''$.
        
        This shows that the group
        $D = \operatorname{coker}(C \to C'')$
        is independent of $j \ge N$.
        From the proof of the proposition, we have
            $
                    D
                \cong
                    \operatorname{coker}(
                            \operatorname{Gr}_{j} \mathscr{K}
                        \to
                            \operatorname{Gr}_{j} \mathscr{K}'
                    )
            $
        for $j \ge N$.
        Hence
            $
                    D
                \cong
                    \operatorname{coker}(
                            \operatorname{Gr} \mathscr{K}
                        \to
                            \operatorname{Gr} \mathscr{K}'
                    )
            $
        as a (pro)algebraic group over $k$,
        where $\operatorname{Gr} = \varprojlim_{j} \operatorname{Gr}_{j}$
        denotes the Greenberg transform of infinite level
        (\cite[Section 14]{BGA}).
\end{remark}

Following \cite[p. 473]{LLR}, we make the following 
\begin{definition} We define the invariant $\chi_{\mathrm{points}}(\mathscr{G}^\bullet)$ as
$$\chi_{\mathrm{points}}(\mathscr{G}^\bullet) : = \sum_{i=1}^n (-1) ^ i \dim_k D_i.$$
\end{definition}
The following result is stated in \cite{LLR} (see \it op. cit., \rm Theorem 2.10) without proof; we shall provide the details because it plays a fundamental role in our argument:
\begin{theorem}
For a complex $\mathscr{G}^\bullet$ as above, we have \label{Liepointsthm}
$$\chi_{\mathrm{Lie}}(\widetilde{\mathscr{G}}^\bullet)=\chi_{\mathrm{points}}(\mathscr{G}^\bullet).$$ 
\end{theorem}
\begin{proof}
Because the map $\widetilde{\mathscr{G}}^\bullet(\Og_K) \to \mathscr{G}^\bullet(\Og_K)$ is an isomorphism of complexes, we may assume without loss of generality that all the $\mathscr{G}^j$ are smooth.
There is nothing to do if $n = 0$ because $ \mathscr{G}^\bullet$ is the zero complex in this case.
In general, we let $\mathscr{K}:=\ker(\mathscr{G}^{n-1}\to \mathscr{G}^n)$ and decompose $\mathscr{G}^\bullet$ into the complexes 
$$\mathscr{H}^\bullet_1: \, 0 \to \mathscr{G}^1 \to ... \to \mathscr{G}^{n-2} \to \mathscr{K} \to 0,$$
as well as $$\mathscr{H}_2^\bullet: \, 0\to \mathscr{K} \to \mathscr{G}^{n-1} \to \mathscr{G}^n\to 0,$$ declaring that $\mathscr{K}$ be placed in degree $n-2$ in $\mathscr{H}^\bullet_2.$
Since the functors $\Gamma(\Og_K,-)$ and $\Lie -$ are left exact, we see the following:
On the one hand, denoting by $D_j^{\mu}$ the group schemes constructed in Proposition \ref{Diconstructprop} for $\mathscr{H}^\bullet_{\mu}$ ($\mu=1,2$), and by $D_j$ those constructed for $\mathscr{G}^\bullet$, we find $D_i=D_i^1$ for $i\leq n-1,$ $D^2_{n-2}=D^2_{n-1}=0,$ and $D_j=D^2_j$ for $j\geq n.$ On the other hand, $H^{i}(\Lie\mathscr{G}^\bullet) = H^{i}(\Lie\mathscr{H}^\bullet_1)$ for $i \leq n-1,$ $H^{n-2}(\Lie\mathscr{H}^\bullet_2)=H^{n-1}(\Lie\mathscr{H}^\bullet_2)$=0, and $H^{j}(\Lie\mathscr{H}^\bullet_2)=H^{j}(\Lie\mathscr{G}^\bullet)$ for $j\geq n.$ Note that, writing $\widetilde{\mathscr{K}}$ for the smoothening of $\mathscr{K},$ the canonical map $\Lie\widetilde{\mathscr{K}} \to \Lie \mathscr{K}$ is injective. Moreover, writing $\delta:=\ell_{\Og_K}(\mathrm{coker}(\Lie\widetilde{\mathscr{K}}\to\Lie\mathscr{K})),$ we see that $\chi_{\mathrm{Lie}}(\widetilde{\mathscr{H}}^\bullet_1) = \chi_{\mathrm{Lie}}(\mathscr{H}_1^\bullet)-(-1)^{n-1}\delta$ and $\chi_{\mathrm{Lie}}(\widetilde{\mathscr{H}}^\bullet_2) = \chi_{\mathrm{Lie}}(\mathscr{H}_2^\bullet)+(-1)^{n-1}\delta.$ This implies that 
$$\chi_{\mathrm{Lie}}(\mathscr{G}^\bullet)=\chi_{\mathrm{Lie}}(\widetilde{\mathscr{H}}_1^\bullet)+\chi_{\mathrm{Lie}}(\widetilde{\mathscr{H}}_2^\bullet),$$ as well as
$$\chi_{\mathrm{points}}(\mathscr{G}^\bullet)=\chi_{\mathrm{points}}(\mathscr{H}_1^\bullet)+\chi_{\mathrm{points}}(\mathscr{H}_2^\bullet)  =\chi_{\mathrm{points}}(\widetilde{\mathscr{H}}_1^\bullet)+\chi_{\mathrm{points}}(\widetilde{\mathscr{H}}_2^\bullet).$$
The equality
	\[
			\chi_{\mathrm{Lie}}(\widetilde{\mathscr{H}}_2^\bullet)
		=
			\chi_{\mathrm{points}}(\widetilde{\mathscr{H}}_2^\bullet)
	\]
is precisely \cite[Theorem 2.1(b)]{LLR}.
Hence the statement for $n$ reduces to the statement for $n - 1$.
By induction, we get the result.
\end{proof}

\subsection{Complexes of Néron models}
Let $G^\bullet:\,0 \to T \to B \to A \to 0$ be an exact sequence of semiabelian varieties over $K$ and let 
$$\mathscr{G}^\bullet\colon \, 0 \to \mathscr{T} \to \mathscr{B}\to \mathscr{A} \to 0$$ be the associated complex of Néron lft-models over $\Og_K$ with $\mathscr{T}$ placed in degree 1. Let $L$ be a finite separable extension of $K$ such that $T,$ $B,$ and $A$ all acquire semiabelian reduction over the integral closure $\Og_L$ of $\Og_K$ in $L,$ and let $\mathscr{T}_L,$ $\mathscr{B}_L,$ and $\mathscr{A}_L$ be the Néron lft-models of $T_L,$ $B_L,$ and $A_L,$ respectively. The induced complex of Néron lft-models will be called $\mathscr{G}_L^\bullet.$ Following Chai \cite{Chai}, we define the \it base change conductor \rm of $B$ as
$$c(B):=\frac{1}{[L:K]} \ell_{\Og_L}(\mathrm{coker}((\Lie \mathscr{B})\otimes_{\Og_K}\Og_L \to \Lie \mathscr{B}_L)),$$ and similarly for $T$ and $A.$ Then we have the following
\begin{conjecture}\rm (Chai \cite{Chai}) \it Assume that $T$ is a torus and that $A$ is an Abelian variety. Then 
$$c(B)=c(T)+c(A).$$
\end{conjecture}
The main difficulty comes from the fact that Néron lft-models do not in general commute with ramified base change, and that the complex $\mathscr{G}^\bullet$ above is usually not exact. We shall see that, in general, the additivity of the base change conductor as in Chai's conjecture (without the hypotheses on $T$ and $A$) is equivalent to a growth condition on the invariant $\gamma(-)=\chi_{\mathrm{Lie}}(-)$, generalising \cite[Proposition 4.1.1]{CLN}:
\begin{proposition}
Let $G^\bullet$ be as above. Then we have 
$$c(B)-c(T)-c(A)=\frac{1}{[L:K]}\gamma(\mathscr{G}_L^\bullet) - \gamma(\mathscr{G}^\bullet) \in \frac{1}{[L:K]}\Z.$$ \label{Chigammaprop}
In particular, the following are equivalent:\\
(i) $c(B)=c(T)+c(A),$\\
(ii) $\gamma(\mathscr{G}_L^\bullet)=[L:K]\gamma(\mathscr{G}^\bullet),$ and\\
(iii) $\chi_{\mathrm{Lie}}(\mathscr{G}_L^\bullet)=[L:K]\chi_{\mathrm{Lie}}(\mathscr{G}^\bullet)$ 
\end{proposition} 
\begin{proof}
Let $\mathfrak{m}_L\subseteq \Og_L$ be the maximal ideal and write $d:=[L:K].$ By the definition of $c(-)$, we have an equality $$\det \Lie \mathscr{T}_L = \mathfrak{m}_L^{-dc(T)}(\det \Lie\mathscr{T}) \otimes_{\Og_K} \Og_L;$$ analogous equalities hold with $T$ replaced by $B$ and $A.$
Using the definition of $\gamma(-),$ we obtain the equalities
\begin{align*}
\mathfrak{m}_L^{-\gamma(\mathscr{G}^\bullet _L)} & = 
\det \Lie \mathscr{G}^\bullet _L \\ & =
(\det \Lie\mathscr{T}_L)^\vee \otimes_{\Og_L} \det \Lie\mathscr{B}_L \otimes_{\Og_L} (\det \Lie\mathscr{A}_L)^\vee \\
&= \mathfrak{m}_L ^{d(c(T)-c(B)+c(A))} (\det \Lie \mathscr{G}^\bullet) \otimes_{\Og_K} \Og_L \\
&= \mathfrak{m}_L ^{d(c(T)-c(B)+c(A)) - d\gamma(\mathscr{G}^\bullet)}
\end{align*}
 of $\Og_L$-submodules of $L.$ This implies claim (i). The equivalence between (i), (ii), and (iii) then follows immediately using Lemma \ref{chigammalem}.
\end{proof}

\begin{remark}\label{Remark210} If $T$ is a torus, $T_L$ is split because $T_L$ has semiabelian reduction and $k$ is algebraically closed. In particular, the complex $\mathscr{G}^\bullet_L$ is exact \cite[Remark 4.8 (a)]{Chai}, so $\gamma(\mathscr{G}^\bullet_L)=0.$ Therefore we recover \cite[Proposition 4.1.1]{CLN}, which says that $(i)$ is equivalent to $\gamma(\Lie \mathscr{G}^\bullet)=0$ in this case.
\end{remark}

We have now assembled the tools required to prove Chai's conjecture. In fact, we shall prove a slightly stronger statement. For any semiabelian variety $B,$ we denote by $B^{\mathrm{b}}$ the quotient of $B$ by its maximal \it split \rm subtorus. 
\begin{theorem} \rm(Chai's conjecture) \it
Let $0\to T\to B\to A\to 0$ be an exact sequence of semiabelian varieties over $K$ and suppose that $T$ is a torus.\label{Chaithm} Then $$c(B)=c(T)+c(A).$$ 
\end{theorem}
\begin{proof}
We shall give two related (but independent) proofs. Let $\mathscr{G}^\bullet$ be the complex $0\to \mathscr{T}\to \mathscr{B}\to \mathscr{A}\to 0$ of Néron lft-models induced by the sequence from the Theorem. Moreover, let $T'$ be the kernel of the map $B^{\mathrm{b}}\to A^{\mathrm{b}}.$ By \cite[Proposition 3.4.2 (1)]{CLN}, we obtain an exact sequence
$$0\to T'\to B^{\mathrm{b}} \to A^{\mathrm{b}} \to 0,$$ as well as an isogeny $T^{\mathrm{b}}\to T'.$ Because $c(B)=c(B^{\mathrm{b}})$  for any semiabelian variety $B$ \cite[Lemma 4.1.3]{CLN}, and because the base change conductor is invariant under isogeny for algebraic tori (\cite[Theorem on p. 367 and Theorem 12.1]{CY} or Theorem \ref{0025} below), we may assume that $T,$ $A,$ and $B$ contain no split subtorus, and hence that their Néron lft-models are of finite type \cite[Chapter 10.2, Theorem 1]{BLR}. Therefore we can apply our previous results to $\mathscr{G}^\bullet.$ Because $k$ is algebraically closed, we have $H^1(K, T)=0$ \cite[Lemma 4.3]{Chai}, so the sequence $0\to T(K)\to B(K)\to A(K)\to 0$ is exact. However, this sequence is canonically identified with the complex $\mathscr{G}^\bullet(\Og_K)$ by the universal property of the Néron lft-model, so we find that 
$\chi_{\mathrm{points}}(\mathscr{G}^\bullet)=0.$ Theorem \ref{Liepointsthm} now shows that $\chi_{\mathrm{Lie}}(\mathscr{G}^\bullet)=0,$ which, in turn, implies that 
$$\gamma(\Lie\mathscr{G}^\bullet)=0$$ (but note that $\chi_{\mathrm{Lie}}(\mathscr{G}^\bullet)=0$ does not imply that $\Lie \mathscr{G}^\bullet$ is exact). However, this is equivalent to the claim that $c(B)=c(T)+c(A)$ by Proposition \ref{Chigammaprop} and the Remark thereafter. Note that, if $A$ has no split subtorus (as is the case in Chai's original conjecture), then the sequence $0\to T^{\mathrm{b}}\to B^{\mathrm{b}}\to A^{\mathrm{b}}\to 0$ is exact \cite[Proposition 3.4.2(2)]{CLN}, so we do not need to invoke the invariance of $c(-)$ under isogeny for tori in that case.

We shall now give a second argument, which allows us to circumvent using the isogeny invariance of $c(-)$ in all cases: Consider the commutative diagram
$$\begin{CD}
&&0&& 0 && 0\\
&&@VVV@VVV@VVV\\
0@>>>\mathscr{T}^{0}(\Og_K)@>>>\mathscr{B}^{0}(\Og_K) @>>> \mathscr{A}^{0}(\Og_K)@>>>0\\
&&@VVV@VVV@VVV\\
0@>>>\mathscr{T}(\Og_K)@>>>\mathscr{B}(\Og_K) @>>> \mathscr{A}(\Og_K)@>>>0\\
&&@VVV@VVV@VVV\\
0@>>> \Phi_T @>>> \Phi_B@>>>\Phi_A @>>>0\\
&&@VVV@VVV@VVV\\
&&0&& 0 && 0,
\end{CD}$$
where $\Phi_T,$ $\Phi_B,$ and $\Phi_A$ are the obvious cokernels. Denoting the first non-zero row of this diagram by $\mathscr{G}^{0, \bullet}(\Og_K)$ and the third by $\Phi^\bullet,$ we obtain the exact triangle
$$\mathscr{G}^{0, \bullet}(\Og_K) \to \mathscr{G}^\bullet(\Og_K) \to \Phi^\bullet \to \mathscr{G}^{0, \bullet}(\Og_K)[1]$$ in the bounded derived category of Abelian groups. This follows from the fact that all columns of the diagram above are exact. However, as in the first proof, we see hat the middle row in the diagram is exact, so it is quasi-isomorphic to the zero object. In particular, the morphism 
$$\Phi^\bullet \to \mathscr{G}^{0, \bullet}(\Og_K)[1]$$ is a quasi-isomorphism. For $i=1,2,3,$ denote by $D_i$ the group schemes constructed in Proposition \ref{Diconstructprop} from the complex $\mathscr{G}^{0,\bullet}.$ By \cite[Proposition 3.5]{HN}, the terms appearing in $\Phi^\bullet$ are finitely generated, and hence so are the $D_i(k)\cong H^ i(\Phi^\bullet[-1]).$ Since $k$ is algebraically closed, the groups $D^0_j(k)$ are $n$-divisible for any $n\in \N$ invertible in $k.$ Since they are also finitely generated, we see that $D_j^0(k)$ is finite for all $j,$ so $D_j^0=0$ for all $j$ since $k$-points are dense in any smooth $k$-scheme. This shows that the $D_j$ are étale and hence $\dim_kD_j=0$ for all $j,$ which implies in particular that $\chi_{\mathrm{points}}(\mathscr{G}^{0,\bullet})=0.$ Because the map $\Lie\mathscr{G}^{0,\bullet}\to \Lie \mathscr{G}^\bullet$ is an isomorphism of complexes, we deduce that $\gamma(\Lie \mathscr{G}^\bullet)=0$ and conclude as in the first proof. 
\end{proof}

It is possible to generalise this result slightly, as follows:
\begin{theorem}
Let $0\to T\to B\to A \to 0$ be an exact sequence of semiabelian varieties over $K.$ Let $\mathscr{G}^\bullet$ be the induced complex of Néron lft-models over $\Og_K.$ \label{generalthm} Then the following are equivalent:\\
(i) $\gamma(\mathscr{G}^\bullet)=0$\\
(ii) the Abelian group 
$$\Delta:=\mathrm{coker}(B(K)\to A(K))=\ker(H ^ 1(K,T) \to H ^1 (K,B))$$
is finite.

In particular, if $B(L) \to A(L)$ is surjective for all finite separable extensions $L$ of $K,$ then $c(B)=c(T)+c(A).$
\end{theorem}
\begin{proof}
Let $$\mathscr{G}^\bullet:\, 0\to\mathscr{T}\to\mathscr{B}\to\mathscr{A}\to 0$$
be the complex of Néron lft-models induced by the sequence from the Theorem. For $i=1,2,3,$ let $D_i$ be the canonical $k$-group scheme such that $H^i(\mathscr{G}^{0,\bullet}(\Og_K))=D_i(k)$ constructed in Proposition \ref{Diconstructprop}. We shall use the same notation as in the proof of Theorem \ref{Chaithm}. Since the sequence
$$0\to T(K)\to B(K) \to A(K)$$ is exact, 
the long exact cohomology sequence induced by the exact triangle
$\mathscr{G}^{0, \bullet}(\Og_K) \to \mathscr{G}^\bullet(\Og_K) \to \Phi^\bullet \to \mathscr{G}^{0, \bullet}(\Og_K)[1]$ gives us the canonical isomorphisms $H^1(\mathscr{G}^{0,\bullet}(\Og_K))=0$ and $H^2(\mathscr{G}^{0,\bullet}(\Og_K))=H^1(\Phi^\bullet),$ as well as a diagram
$$\begin{CD}H^2(\Phi^\bullet)@>>> H^3(\mathscr{G}^{0,\bullet}(\Og_K))@>>> H^3(\mathscr{G}^\bullet(\Og_K))@>>> H^3(\Phi^\bullet)\\
&&@V{\cong}VV@VV{\cong}V\\
&&D_3(k)&& \Delta.
\end{CD}$$
As in the proof of Theorem \ref{Chaithm}, we see that the $H^i(\Phi^\bullet)$ are finitely generated for all $i,$ so that (ii) is true if and only if $D_3(k)$ is finitely generated (because $\Delta$ is torsion), which happens if and only if $\dim_k D_3=0.$
Moreover, since $D_1(k)=0$ and $D_2(k)$ is finitely generated, we see that 
$$-\gamma(\mathscr{G}^\bullet)=\dim_k D_3.$$ These observations together imply the equivalence $(i) \iff (ii).$ The last claim now follows from Proposition \ref{Chigammaprop}.
\end{proof}

\begin{remark} We shall see later that the first proof we gave above cannot be generalised to give another proof of Theorem \ref{generalthm}. The reason is that the invariance of $c(-)$ under isogeny does not generalise to all semiabelian varieties.
\end{remark}

\section{The ind-rational pro-\'etale site}
\label{Indrat-Proet-Section}
We will recall from \cite{Suz} some notation and facts
about the ind-rational pro-\'etale site and
ind- and pro-algebraic groups.

As before, let $k$ be an algebraically closed
(or, more generally, a perfect) field of characteristic $p > 0$.
As in \cite[Section 2.1]{Suz},
a $k$-algebra is said to be \emph{rational}
if it can be written as a finite product
$k_{1}' \times \dots \times k_{n}'$,
where each $k_{i}'$ is the perfection
(direct limit along Frobenius)
of a finitely generated field over $k$.
A $k$-algebra is said to be \emph{ind-rational}
if it is a filtered direct limit of rational $k$-algebras.
Define a Grothendieck site
$\Spec k^{\mathrm{indrat}}_{\mathrm{proet}}$
to be (the opposite of) the category of ind-rational $k$-algebras
with $k$-algebra homomorphisms endowed with the pretopology
where a covering $\{k' \to k_{i}'\}_{i}$ is
a finite family of ind-\'etale homomorphisms
such that $k' \to \prod_{i} k_{i}'$ is faithfully flat.
Let $\operatorname{Ab}(k^{\mathrm{indrat}}_{\mathrm{proet}})$
be the category of sheaves of abelian groups
on $\Spec k^{\mathrm{indrat}}_{\mathrm{proet}}$.
Let $\operatorname{\mathbf{Hom}}_{k^{\mathrm{indrat}}_{\mathrm{proet}}}$
be the sheaf-Hom functor for
$\operatorname{Ab}(k^{\mathrm{indrat}}_{\mathrm{proet}})$
and 
$\operatorname{\mathbf{Ext}}^{n}_{k^{\mathrm{indrat}}_{\mathrm{proet}}}$
its $n$-th right derived functor.

Let $\mathrm{Alg} / k$ be the category of perfections
(inverse limits along Frobenius)
of commutative algebraic groups over $k$
with $k$-group scheme morphisms.
(Here an algebraic group over $k$ means
a quasi-compact smooth group scheme over $k$.)
Let $\mathrm{PAlg} / k$ be the pro-category of $\mathrm{Alg} / k$,
$\mathrm{IAlg} / k$ the ind-category of $\mathrm{Alg} / k$
and $\mathrm{IPAlg} / k$ the ind-category of $\mathrm{PAlg} / k$.
The endofunctors $(\;\cdot\;)^{0}$ (identity component)
and $\pi_{0}$ (component group) on $\mathrm{Alg} / k$
naturally extend to endofunctors on $\mathrm{IPAlg} / k$.
An object $G \in \mathrm{IPAlg} / k$ is said to be \emph{connected}
if $G^{0} \overset{\sim}{\to} G$.
The natural functors
    \[
        \begin{CD}
                \mathrm{Alg} / k
            @>>>
                \mathrm{PAlg} / k
            @.
            \\ @VVV @VVV @. \\
                \mathrm{IAlg} / k
            @>>>
                \mathrm{IPAlg} / k
            @>>>
                \operatorname{Ab}(k^{\mathrm{indrat}}_{\mathrm{proet}})
        \end{CD}
    \]
are fully faithful exact by \cite[Proposition (2.3.4)]{Suz},
where the right lower functor is the Yoneda functor.
If an object $G \in \mathrm{Alg} / k$ is the perfection of a unipotent group,
then by \cite[Proposition (2.4.1) (b)]{Suz},
the sheaf
$\operatorname{\mathbf{Hom}}_{k^{\mathrm{indrat}}_{\mathrm{proet}}}
(G, \Q / \Z)$
agrees with the Pontryagin dual of $\pi_{0}(G)$,
the sheaf
$\operatorname{\mathbf{Ext}}^{1}_{k^{\mathrm{indrat}}_{\mathrm{proet}}}
(G, \Q / \Z)$
agrees with the Breen-Serre dual
(\cite[Chapter III, Lemma 0.13]{Milne}) of $G^{0}$
and the sheaf
$\operatorname{\mathbf{Ext}}^{n}_{k^{\mathrm{indrat}}_{\mathrm{proet}}}
(G, \Q / \Z)$
vanishes for $n \ge 2$.
The Breen-Serre dual of $G$ has the same dimension as $G$
by \cite[Proposition 1.2.1]{Beg}.

Any object $G \in \mathrm{IAlg} / k$ commutes with filtered direct limits
as a functor from the category of ind-rational $k$-algebras
to the category of abelian groups.
In other words, $G$ is locally of finite presentation as a sheaf;
see the paragraphs after \cite[Proposition (2.4.1)]{Suz}
for more details about this notion.
In particular, $G$ is determined by values on rational $k$-algebras
and hence by values on perfections of finitely generated fields over $k$,
with
    \[
            G(k')
        \cong
            G(\overline{k'})^{\Gal(\overline{k'} / k')}
    \]
functorially in perfect fields $k'$ over $k$.

As before, let $K$ be a complete discrete valuation field
with ring of integers $\mathcal{O}_{K}$ and residue field $k$.
We have a canonical $W(k)$-algebra structure on $\mathcal{O}_{K}$,
where $W$ denotes the functor of $p$-typical Witt vectors of infinite length.
This structure $W(k) \to \mathcal{O}_{K}$ factors through $k$
if $K$ has equal characteristic.
As in \cite[Section 3.1]{Suz},
for any ind-rational $k$-algebra $k'$,
define
    \begin{gather*}
                \mathbf{O}_{K}(k')
            =
                \varprojlim_{n \ge 1}
                    \bigl(
                        W(k') \otimes_{W(k)} \mathcal{O}_{K} / \mathfrak{m}^{n}
                    \bigr),
        \\
                \mathbf{K}(k')
            =
                \mathbf{O}_{K}(k') \otimes_{\mathcal{O}_{K}} K.
    \end{gather*}
If $k'$ is a (perfect) field,
then the ring $\mathbf{O}_{K}(k')$ is a complete discrete valuation ring
with residue field $k'$ and
ramification index $1$ over $\mathcal{O}_{K}$
in the sense of \cite[Section 3.6, Definition 1]{BLR}.

For any fppf sheaf of abelian groups $G$ on $\mathcal{O}_{K}$
and any integer $n$, define a sheaf
$\mathbf{H}^{n}(\mathcal{O}_{K}, G)$
on $\Spec k^{\mathrm{indrat}}_{\mathrm{proet}}$
to be the sheafification of the presheaf
    \[
            k'
        \mapsto
            H^{n}(\mathbf{O}_{K}(k'), G),
    \]
where $k'$ runs through ind-rational $k$-algebras.%
\footnote{
    The sheafification here is the pro-\'etale sheafification.
    The notation in \cite[Section 3.3]{Suz} is actually
    $\Tilde{\mathbf{H}}^{n}(\mathcal{O}_{K}, G)$.
    The notation $\mathbf{H}^{n}(\mathcal{O}_{K}, G)$
    in \cite[Section 3.3]{Suz} is reserved for
    the \'etale sheafification of the same presheaf.
    But the \'etale sheafification is already a pro-\'etale sheaf
    under the ``P-acyclicity'' condition (\cite[Section 2.4]{Suz}),
    which is practically almost always satisfied (\cite[Section 3.4]{Suz}).
    In this paper, we exclusively use the pro-\'etale sheafified version.
}
Set $\mathbf{\Gamma} = \mathbf{H}^{0}$.
If $k'$ is an algebraically closed field extension of $k$,
then
    \[
            \mathbf{H}^{n}(\mathcal{O}_{K}, G)(k')
        \cong
            H^{n}(\mathbf{O}_{K}(k'), G)
    \]
(that is, the sheafification is unnecessary for $k'$-valued points).

Similarly, for any fppf sheaf of abelian groups $G$ on $K$,
define a sheaf
$\mathbf{H}^{n}(K, G)$
on $\Spec k^{\mathrm{indrat}}_{\mathrm{proet}}$
to be the sheafification of the presheaf
    \[
            k'
        \mapsto
            H^{n}(\mathbf{K}(k'), G).
    \]
Set $\mathbf{\Gamma} = \mathbf{H}^{0}$.
If $k'$ is an algebraically closed field extension of $k$,
then
    \[
            \mathbf{H}^{n}(K, G)(k')
        \cong
            H^{n}(\mathbf{K}(k'), G).
    \]

We recall the following results:

\begin{proposition}[{\cite[Proposition (3.4.3) (b)]{Suz}}] \label{SuzFinFlatProp}
    Let $N$ be a commutative finite flat group scheme over $K$.
    \begin{enumerate}
        \item
            The sheaf $\mathbf{\Gamma}(K, N)$ is
            a finite \'etale group scheme over $k$.
        \item
            The sheaf $\mathbf{H}^{1}(K, N)$ is
            an object of $\mathrm{IPAlg} / k$.
        \item
            We have $\mathbf{H}^{n}(K, N) = 0$ for $n \ge 2$.
    \end{enumerate}
\end{proposition}

\begin{proposition}[{\cite[Propositions (3.4.2) (b) and (3.4.6)]{Suz}}] \label{FinFlatOK}
    Let $N$ be a commutative finite flat group scheme over $\mathcal{O}_{K}$.
    \begin{enumerate}
        \item
            The sheaf $\mathbf{\Gamma}(\mathcal{O}_{K}, N)$ is
            a finite \'etale group scheme over $k$.
        \item
            The sheaf $\mathbf{H}^{1}(\mathcal{O}_{K}, N)$ is
            an object of $\mathrm{PAlg} / k$ that is connected.
        \item
            The sheaf
            $\mathbf{H}^{1}(K, N) / \mathbf{H}^{1}(\mathcal{O}_{K}, N)$
            is an object of $\mathrm{IAlg} / k$.
        \item
            We have $\mathbf{H}^{n}(\mathcal{O}_{K}, N) = 0$ for $n \ge 2$.
    \end{enumerate}
\end{proposition}

\begin{proposition}
\label{FinDuality}
    Assume that $K$ has equal characteristic.
    Let $N$ be a commutative finite flat group scheme over $\mathcal{O}_{K}$
    with Cartier dual $M$.
    Then there exist canonical isomorphisms
        \begin{gather*}
                    \mathbf{H}^{1}(K, N)^{0} / \mathbf{H}^{1}(\mathcal{O}_{K}, N)
                \cong
                    \operatorname{\mathbf{Ext}}^{1}_{k^{\mathrm{indrat}}_{\mathrm{proet}}}
                    \bigl(
                        \mathbf{H}^{1}(\mathcal{O}_{K}, M),
                        \Q / \Z
                    \bigr),
                \\
                    \pi_{0}(\mathbf{H}^{1}(K, N))
                \cong
                    \operatorname{\mathbf{Hom}}_{k^{\mathrm{indrat}}_{\mathrm{proet}}}
                    \bigl(
                        \mathbf{\Gamma}(K, M),
                        \Q / \Z
                    \bigr).                    
        \end{gather*}
    In particular, $\pi_{0}(\mathbf{H}^{1}(K, N))$ is finite \'etale.
\end{proposition}

\begin{proof}
    This is \cite[Theorem (5.2.1.2)]{Suz}.
    More explicitly, this theorem says that
    we have a spectral sequence
        \[
                E_{2}^{i j}
            =
                \operatorname{\mathbf{Ext}}^{i}_{k^{\mathrm{indrat}}_{\mathrm{proet}}}(\mathbf{H}^{2 - j}(\mathcal{O}_{K}, M), \Q / \Z)
            \Longrightarrow
                \mathbf{H}^{i + j}_{x}(\mathcal{O}_{K}, N),
        \]
    where the object $\mathbf{H}^{n}_{x}(\mathcal{O}_{K}, N)$ on the right is
    $\mathbf{H}^{1}(K, N) / \mathbf{H}^{1}(\mathcal{O}_{K}, N)$
    for $n = 2$ and zero otherwise
    (\cite[Proposition (3.4.6)]{Suz} and its proof).
    By Proposition \ref{FinFlatOK},
    we have $E_{2}^{i j} = 0$ unless
    $(i, j) = (0, 2)$ or $(1, 1)$.
    Hence this spectral sequence reduces to an exact sequence
        \[
                0
            \to
                \operatorname{\mathbf{Ext}}^{1}_{k^{\mathrm{indrat}}_{\mathrm{proet}}}(\mathbf{H}^{1}(\mathcal{O}_{K}, M), \Q / \Z)
            \to
                \mathbf{H}^{1}(K, N) / \mathbf{H}^{1}(\mathcal{O}_{K}, N)
            \to
                \operatorname{\mathbf{Hom}}_{k^{\mathrm{indrat}}_{\mathrm{proet}}}(\mathbf{\Gamma}(\mathcal{O}_{K}, M), \Q / \Z)
            \to
                0.
        \]
    The first (non-zero) term is connected
    by \cite[Proposition (2.4.1) (b)]{Suz}.
    The third term is finite, and
    $\mathbf{\Gamma}(\mathcal{O}_{K}, M) = \mathbf{\Gamma}(K, M)$
    (\cite[Proposition (3.4.2) (b)]{Suz}).
    Hence this exact sequence is a connected-\'etale sequence.
    Since $\mathbf{H}^{1}(\mathcal{O}_{K}, N)$ is connected
    (Proposition \ref{FinFlatOK}),
    we have
        \[
                \bigl(
                    \mathbf{H}^{1}(K, N) / \mathbf{H}^{1}(\mathcal{O}_{K}, N)
                \bigr)^{0}
            =
                \mathbf{H}^{1}(K, N)^{0} / \mathbf{H}^{1}(\mathcal{O}_{K}, N),
        \]
        \[
                \pi_{0} \bigl(
                    \mathbf{H}^{1}(K, N) / \mathbf{H}^{1}(\mathcal{O}_{K}, N)
                \bigr)
            =
                \pi_{0}(\mathbf{H}^{1}(K, N)).
        \]
    Hence the above exact sequence induces the desired isomorphisms.
\end{proof}

\begin{proposition}[{\cite[Proposition (3.4.3) (d)]{Suz}}]
    Let $A$ be an Abelian variety over $K$.
    \begin{enumerate}
        \item
            The sheaf $\mathbf{\Gamma}(K, A)$ is
            an object of $\mathrm{PAlg} / k$.
            It is the perfection of the Greenberg transform of infinite level
            of the N\'eron model of $A$.
        \item
            The sheaf $\mathbf{H}^{1}(K, A)$ is
            an object of $\mathrm{IAlg} / k$.
        \item
            We have $\mathbf{H}^{n}(K, A) = 0$ for $n \ge 2$.
    \end{enumerate}
\end{proposition}

\begin{proposition}[{\cite[Theorem (7.2)]{Suz}}] \label{SuzDualityProp}
    Let $A$ be an Abelian variety over $K$ with dual $B$.
    Then there exists a canonical isomorphism
        \[
                \mathbf{H}^{1}(K, A)
            \cong
                    \operatorname{\mathbf{Ext}}^{1}_{k^{\mathrm{indrat}}_{\mathrm{proet}}}
                    \bigl(
                        \mathbf{\Gamma}(K, B),
                        \Q / \Z
                    \bigr).
        \]
\end{proposition}

Below we give a few more.

\begin{proposition}
\label{FinDualMixed}
    The statements of Proposition \ref{FinDuality} remain true
    even if $K$ has mixed characteristic.
\end{proposition}

\begin{proof}
    This is basically a translation of B\'egueri's results \cite{Beg}.
    More precisely, in \cite[Section 4.2]{Beg},
    she gives $H^{1}(\mathcal{O}_{K}, N)$ a structure
    as the perfection of an algebraic group over $k$,
    which is defined as the cokernel of
    $\operatorname{Gr}^{\mathrm{perf}}(G^{1})
    \to \operatorname{Gr}^{\mathrm{perf}}(G^{2})$,
    where $0 \to N \to G^{1} \to G^{2} \to 0$ is a resolution of $N$
    by commutative smooth group schemes over $\mathcal{O}_{K}$
    and $\operatorname{Gr}^{\mathrm{perf}}$ is
    the perfection of the Greenberg transform of infinite level.
    By \cite[Proposition (3.4.2)]{Suz} and its proof,
    we can see that this algebraic structure agrees with
    our $\mathbf{H}^{1}(\mathcal{O}_{K}, N)$.
    Similarly, in \cite[Section 4.3]{Beg},
    her algebraic structure for $H^{1}(K, N)$ is defined as
    the cokernel of
    $\operatorname{Gr}^{\mathrm{perf}}(\mathscr{T}^{1})
    \to \operatorname{Gr}^{\mathrm{perf}}(\mathscr{T}^{2})$,
    where $0 \to N_{K} \to T^{1} \to T^{2} \to 0$ is a resolution
    of the generic fiber $N_{K}$ by tori
    and $\mathscr{T}^{i}$ is the N\'eron lft-model of $T^{i}$.
    By \cite[Proposition (3.4.3) (e)]{Suz},
    this agrees with our $\mathbf{H}^{1}(K, N)$.
    Then our statements correspond to
    \cite[Proposition 6.1.2 and Theorem 6.3.2]{Beg}.
\end{proof}

\begin{proposition}
\label{EtCriterion}
    Let $N$ be a commutative finite flat group scheme over $\mathcal{O}_{K}$.
    Assume that its generic fiber $N_{K}$ is \'etale.
    Then $\mathbf{H}^{1}(\mathcal{O}_{K}, N)$ is an object of $\mathrm{Alg} / k$.
    Its dimension is the $\mathcal{O}_{K}$-length of
    the pullback along the zero section
    of $\Omega^{1}_{N / \mathcal{O}_{K}}$.
    In particular, $\mathbf{H}^{1}(\mathcal{O}_{K}, N)$ is zero
    if and only if $N$ is \'etale (over $\mathcal{O}_{K}$).
\end{proposition}

\begin{proof}
    As in the proof of Proposition \ref{FinDualMixed},
    we can see that $\mathbf{H}^{1}(\mathcal{O}_{K}, N)$ agrees with
    the perfection of the algebraic structure on $H^{1}(\mathcal{O}_{K}, N)$
    defined in \cite[Section 16]{BGA}.
    Therefore the result follows from
    \cite[Theorem 16.3 and Proposition 16.6]{BGA}.
\end{proof}

\begin{proposition}
\label{LocCohStr}
    Let $N$ be a commutative finite flat group scheme over $\mathcal{O}_{K}$.
    Then $\mathbf{H}^{1}(K, N)^{0} / \mathbf{H}^{1}(\mathcal{O}_{K}, N)$
    is a direct limit indexed by $\N$ of
    perfections of connected unipotent algebraic groups over $k$
    with injective transition morphisms.
\end{proposition}

\begin{proof}
    Let $M$ be the Cartier dual of $N$.
    By \cite[Proposition 16.1 (ii) and Lemma 16.2]{BGA},
    the group $\mathbf{H}^{1}(\mathcal{O}_{K}, M)$ is
    an inverse limit $\varprojlim_{n \ge 1} G_{n}$ of perfections of
    connected unipotent algebraic groups over $k$
    such that each transition morphism
    $\varphi_{n} \colon G_{n + 1} \to G_{n}$
    is surjective with connected kernel.
    Hence
        \[
                \operatorname{\mathbf{Ext}}^{1}_{k^{\mathrm{indrat}}_{\mathrm{proet}}}
                \bigl(
                    \mathbf{H}^{1}(\mathcal{O}_{K}, M),
                    \Q / \Z
                \bigr)
            \cong
                \varinjlim_{n \ge 1}
                \operatorname{\mathbf{Ext}}^{1}_{k^{\mathrm{indrat}}_{\mathrm{proet}}}
                (G_{n}, \Q / \Z)
        \]
    by \cite[Proposition (2.3.3) (c)]{Suz}.
    Each term 
    $\operatorname{\mathbf{Ext}}^{1}_{k^{\mathrm{indrat}}_{\mathrm{proet}}}
    (G_{n}, \Q / \Z)$
    is the perfection of a connected unipotent algebraic group
    by \cite[Proposition (2.4.1) (b)]{Suz}.
    Let $G_{n}' = \Ker(\varphi_{n})$.
    Then we have an exact sequence
        \[
                0
            \to
                \operatorname{\mathbf{Ext}}^{1}_{k^{\mathrm{indrat}}_{\mathrm{proet}}}
                (G_{n}, \Q / \Z)
            \to
                \operatorname{\mathbf{Ext}}^{1}_{k^{\mathrm{indrat}}_{\mathrm{proet}}}
                (G_{n + 1}, \Q / \Z)
            \to
                \operatorname{\mathbf{Ext}}^{1}_{k^{\mathrm{indrat}}_{\mathrm{proet}}}
                (G_{n}', \Q / \Z)
            \to
                0
        \]
    by \cite[Proposition (2.4.1) (a)]{Suz}.
    Now the result follows from Propositions \ref{FinDuality}
    and \ref{FinDualMixed}.
\end{proof}

The following proposition requires $k$ to be algebraically closed.

\begin{proposition}
\label{MultCriterion}
    Let $N$ be a commutative finite flat group scheme over $\mathcal{O}_{K}$.
    Assume that its generic fiber is multiplicative.
    Then $H^{1}(K, N) / H^{1}(\mathcal{O}_{K}, N)$ is finite
    if and only if $N$ is multiplicative (over $\mathcal{O}_{K}$).
\end{proposition}

\begin{proof}
    Let $M$ be the Cartier dual of $N$.
    Since $\pi_{0}(\mathbf{H}^{1}(K, N))$ is always finite \'etale
    by Proposition \ref{FinDuality},
    the finiteness of $H^{1}(K, N) / H^{1}(\mathcal{O}_{K}, N)$
    is equivalent to the finiteness of
    $\bigl(
        \mathbf{H}^{1}(K, N)^{0} / \mathbf{H}^{1}(\mathcal{O}_{K}, N)
    \bigr)(k)$.
    By Proposition \ref{LocCohStr},
    this in turn is equivalent to the vanishing of
    $\mathbf{H}^{1}(K, N)^{0} / \mathbf{H}^{1}(\mathcal{O}_{K}, N)$.
    By Propositions \ref{FinDuality} and \ref{FinDualMixed}
    and the connectedness of $\mathbf{H}^{1}(\mathcal{O}_{K}, M)$,
    this in turn is equivalent to
    $\mathbf{H}^{1}(\mathcal{O}_{K}, M) = 0$,
    which is equivalent to $M$ being \'etale
    by Proposition \ref{EtCriterion}.
    Cartier duality then gives the result.
\end{proof}

\section{Examples}
\subsection{Additivity of $c(-)$ does not imply $\gamma=0$}
Suppose $0\to T\to B \to A\to 0$ is an exact sequence of semiabelian varieties over $K,$ and denote by
$$\mathscr{G}^\bullet:\; 0\to \mathscr{T}\to\mathscr{B}\to\mathscr{A}\to 0$$ the induced complex of Néron lft-models over $\Og_K.$ If $T$ is a torus, we have seen that Chai's conjecture is equivalent to the vanishing of $\gamma(\mathscr{G}^\bullet)$ (see Remark \ref{Remark210}).
Chai \cite{Chai}, as well as Cluckers-Loeser-Nicaise \cite{CLN}, have posed the question whether this still holds if we drop the assumption that $T$ be a torus. As pointed out on p. 907 of \cite{CLN}, this latter claim implies that, if $T,$ $B,$ and $A$ all have semiabelian reduction, then $\gamma(\mathscr{G} ^ \bullet)=0.$ In this subsection, we shall construct a family of counterexamples to this claim. More precisely, we shall construct complete discrete valuation rings $\Og_K$ with algebraically closed residue field and exact sequences $0\to T \to B \to A \to 0$ of semiabelian varieties over $K$ with \it semiabelian reduction \rm over $\Og_K$ such that $\gamma(\mathscr{G}^\bullet)\not=0.$ Our construction will produce examples in arbitrary (mixed or equal) characteristic as long as the residue characteristic is positive. Note that the assumption on the reduction of our semiabelian varieties implies that all three base change conductors vanish, so the equality $c(B)=c(T)+c(A)$ is trivially verified.

We begin with a complete discrete valuation ring $\Og_K$ with algebraically closed residue field $k$ and write $p:=\mathrm{char}\; k.$ Let $\mathscr{F}$ be a finite flat commutative group scheme over $\Og_K$ with multiplicative generic fibre $\mathscr{F}_K$ and of order $p^n$ for some $n\in \N.$

Moreover, we choose $\mathscr{F}$ such that it is not multiplicative globally. Note that, if $\mathrm{char}\; K >0,$ such an $\mathscr{F}$ always exists; for example one could take an anisotropic torus $T$ over $K$ and denote by $\mathscr{T}$ the Néron model of $T.$ For any $n>0,$ the scheme-theoretic closure of $T[p^n]$ in $\mathscr{T}$ has the desired properties. In general, one can choose an elliptic curve $\mathscr{E}\to\Spec\Og_K$ which is generically ordinary and has supersingular special fibre.  Then $\mathscr{E}_K[p^n]$ has a multiplicative subgroup scheme, the scheme-theoretic closure of which in $\mathscr{E}$ again has the desired properties. 

By assumption, we can choose a closed immersion $\mathscr{F}_K\to T$ for some algebraic torus $T$ over $\Og_K.$ Moreover, by \cite[Theorem A.6]{Milne}, there exists an Abelian scheme $\mathscr{A}\to\Spec \Og_K$ and a closed immersion $\mathscr{F}\to\mathscr{A}.$ We obtain exact sequences $0 \to \mathscr{F}_K\to T\to T'\to 0$ over $K$ and $ 0\to \mathscr{F}\to \mathscr{A}\to \mathscr{A}' \to 0$ over $\Og_K,$ where $T'$ and $\mathscr{A}'$ are the obvious cokernels (both of which exist by \cite[Théorème 4.C]{An}). Now define the semiabelian variety $B$ such that it fits into the push-out diagram
\begin{align}\begin{CD} \label{Diag1}
0 @>>> \mathscr{F}_K@>>>\mathscr{A}_K@>>>\mathscr{A}'_K@>>>0\\
&&@VVV@VVV@VV{=}V\\
0@>>>T@>>>B@>>>\mathscr{A}'_K@>>>0.
\end{CD}\end{align}
Because $H^1(K,T)=0=H ^ 2(K,T)$ \cite[Lemma 4.3]{Chai}, the map $H ^ 1(K,B)\to H ^ 1(K, \mathscr{A}'_K)$ is an isomorphism. 
Moreover, comparing the fppf-cohomology sequences induced by $ 0\to \mathscr{F}\to \mathscr{A}\to \mathscr{A}' \to 0$ and $ 0\to \mathscr{F}_K\to \mathscr{A}_K\to \mathscr{A}'_K \to 0$ yields an exact sequence
$$0\to H^1(\Og_K, \mathscr{F})\to H ^ 1(K, \mathscr{F}_K)\to H ^ 1(K, \mathscr{A}_K) \to H ^ 1(K, \mathscr{A}_K').$$
In particular, we find
\begin{align}\ker (H ^ 1(K,\mathscr{A}_K)\to H ^1(K,B)) \nonumber &= \ker (H ^ 1(K,\mathscr{A}_K)\to H ^1(K,\mathscr{A}_K'))\\ \label{sequence2}
&=H ^ 1(K, \mathscr{F}_K)/H ^ 1(\Og_K, \mathscr{F}).
\end{align}
Now observe that, since all group schemes appearing in our construction are commutative, $B$ also fits into the push-out diagram
\begin{align}\begin{CD} \label{Diag3}
0 @>>> \mathscr{F}_K@>>>T @>>>T'@>>>0\\
&&@VVV@VVV@VV{=}V\\
0@>>>\mathscr{A}_K@>>>B@>>>T'@>>>0.
\end{CD}\end{align}
Now we have
\begin{proposition}
Let $\mathscr{G}^\bullet$ be the complex of Néron lft-models induced by the bottom row of diagram (\ref{Diag3}). Then $\gamma(\mathscr{G}^\bullet)\not=0.$ 
\end{proposition}
\begin{proof}
Using the bottom sequence of diagram (\ref{Diag3}) as well as equation (\ref{sequence2}), we find that
\begin{align*}
\mathrm{coker}(B(K)\to T'(K))&=\ker (H ^ 1(K,\mathscr{A}_K)\to H ^1(K,B))\\
&=H ^ 1(K, \mathscr{F}_K)/H ^ 1(\Og_K, \mathscr{F}),
\end{align*}
which is infinite by Proposition \ref{MultCriterion} and our choice of $\mathscr{F}.$ Using the same notation and reasoning as in the proof of Theorem \ref{generalthm}, we find that $\dim_k D_1=\dim_k D_2 = 0$ but $\dim_k D_3>0.$ In particular, if $\mathscr{G}^\bullet$ denotes the complex of Néron lft-models induced by the bottom row of diagram (\ref{Diag3}), then $\gamma(\mathscr{G}^\bullet)\not=0$ by Theorem \ref{Liepointsthm}.
\end{proof}

Finally, observe that the properties of $\mathscr{F},$ $\mathscr{A},$ $T,$ and the morphisms between those objects remain true if we replace $K$ by a finite extension. In particular, we may assume without loss of generality that $T$ is a \it split \rm torus. In this case, the bottom row of diagram (\ref{Diag1}) shows that $T,$ $B,$ and $\mathscr{A}'_K$ all have semiabelian reduction.

The following result, which we record as it might be of general interest, also shows that the step of taking a finite extension in the construction above cannot be removed:
\begin{theorem}
Assume that $\Og_K$ is of mixed characteristic and that $\langle p \rangle = \mathfrak{m}_K^e$ for some $e<p-1.$ Let $0\to T\to B\to A\to 0$ be an exact sequence of semiabelian varieties over $K$ with semiabelian reduction over $\Og_K.$ Then the induced complex
$$\mathscr{G}^ \bullet: \, 0\to \mathscr{T}\to\mathscr{B}\to \mathscr{A}\to 0$$ is exact at $\mathscr{T}$ and $\mathscr{B},$ and the map $\mathscr{B}/\mathscr{T}\to\mathscr{A}$ is an open immersion. In particular, $\gamma(\mathscr{G}^ \bullet)=0.$
\end{theorem}
\begin{proof}
This result is proven in \cite[Chapter 7.5, Theorem 4(ii)]{BLR} in the case where $T,$ $B,$ and $A$ are all Abelian varieties. The same proof can be taken \it mutatis mutandis \rm as long as we can show that, for all $n\geq 0,$ the quasi-finite flat group schemes $\mathscr{T}^0[p^n]/\mathscr{T}^0[p^n]^ \mathrm{f}$ are generically constant, where $(-)^\mathrm{f}$ denotes the finite part of a quasi-finite scheme over a Henselian local base \cite[Tag 04GG (13)]{Stacks}. Note that, since $T$ has semiabelian reduction, there exist $d\in \N,$ an Abelian variety $E$ over $K$ with semiabelian reduction, and a exact sequence
$$0\to \mathbf{G}^d_{\mathrm{m}}\to T\to E\to 0.$$
Let $\mathscr{E}$ be the Néron model of $E.$ Because the component group of the Néron lft-model of $\mathbf{G}^d_{\mathrm{m}}$ is torsion-free, we obtain an exact sequence
$0\to \mathbf{G}^d_{\mathrm{m}}\to \mathscr{T} ^ 0\to \mathscr{E}^ 0 \to 0$ over $\Og_K,$ which induces exact sequences
$$0\to \boldsymbol{\mu}_{p^n}^d \to \mathscr{T} ^ 0[p^n]\to \mathscr{E}^ 0[p^n] \to 0.$$ Taking the pullback of this sequence along the map $\mathscr{E}^0[p^n]^{\mathrm{f}} \to \mathscr{E}^0[p^n]$ shows that there are exact sequences
$$0\to \boldsymbol{\mu}_{p^n}^d \to \mathscr{T} ^ 0[p^n]^{\mathrm{f}}\to \mathscr{E}^ 0[p^n] ^{\mathrm{f}} \to 0$$
for all $n.$ Now the snake lemma shows that the map
$$\mathscr{T} ^ 0[p^n]/\mathscr{T} ^ 0[p^n]^{\mathrm{f}}\to \mathscr{E}^ 0[p^n]/\mathscr{E}^ 0[p^n] ^{\mathrm{f}}$$ is an isomorphism, so the claim follows from Grothendieck's orthogonality theorem as in \it loc. cit.\rm
\end{proof} 

Finally, we shall construct counterexamples to the most natural generalisation of Chai's conjecture, which has been proposed (as a question) in \cite[p. 733]{Chai}, as well as \cite[Question 2.4.1]{CLN}: Given an appropriate choice of $\Og_K$, we shall show that there are exact sequences 
$$0 \to T\to B\to A\to 0$$ of semiabelian varieties over $K$ such that $$c(B)\not=c(T)+c(A).$$ In one of our counterexamples, the base ring will even satisfy both hypotheses the disjunction of which Chai calls the \it awkward assumption \rm \cite[p. 733]{Chai}.

\subsection{Non-additivity in general: Construction I} \label{consIsubs}
Let $\Og_K$ be a complete discrete valuation ring with algebraically closed residue field $k.$ Suppose we can find the following objects:
\begin{itemize}
    \item An Abelian variety $A$ over $K$ with good reduction over $\Og_K,$ and
    \item a finite Galois extension $L$ of $K$ such that the kernel of the map $H ^ 1 (K,A) \to H ^ 1 (L, A_L)$ is infinite. 
\end{itemize}
Starting from these objects, we construct the exact sequence
\begin{align}
0\to A \to \Res_{L/K}A_L \to C \to 0, \label{Rseq}
\end{align}
where $C$ is the cokernel of the closed immersion $A \to \Res_{L/K}A_L.$ 
\begin{proposition} \label{ConIProp}
Under the hypotheses above, we have $c(A)=0$ and $$c(C) < c(\Res_{L/K}A_L).$$
\end{proposition}
\begin{proof}
Let $\mathscr{A}$ be the Néron model of $A$ over $\Og_K$ and let $\Og_L$ be the integral closure of $\Og_K$ in $L.$ Because $A$ has good reduction, we know that $\mathscr{A}_{\Og_L}:=\mathscr{A}\times_{\Og_K}\Og_L$ is the Néron model of its generic fibre; in particular $\Res_{\Og_L/\Og_K} \mathscr{A}_{\Og_L}$ is the Néron model of $\Res_{L/K}A_L.$ Let $\mathscr {C}$ be the Néron model of $C$ over $\Og_K$ and let
$$\mathscr{G}^\bullet: \, 0\to \mathscr{A} \to \Res_{\Og_L/\Og_K} \mathscr{A}_{\Og_L} \to  \mathscr{C} \to 0 $$
be the induced complex of Néron models. Note that $c(A)=0$ because $A$ has good reduction. We shall give two proofs.

\it First proof: \rm Note that the sequence (\ref{Rseq}) becomes split over $L,$ so the induced sequence $\mathscr{G}_L^\bullet$ of Néron models over $\Og_L$ is exact. In particular, $\gamma(\mathscr{G}_L^\bullet)=0.$ Let $D_3$ be the group scheme constructed in Proposition \ref{Diconstructprop} such that 
$$H^3(\mathscr{G}^\bullet(\Og_K)) = D_3(k).$$ Using  Proposition \ref{Chigammaprop}, Theorem \ref{Liepointsthm}, and the fact that $H^ i(\mathscr{G}^\bullet(\Og_K))=0$ for $i \not=3,$ we see that
$$c(\Res_{L/K} A) - c(C) = -\gamma(\mathscr{G^\bullet})= -\chi_{\mathrm{points}}(\mathscr{G}^\bullet) = \dim_k D_3.$$ Because $D_3$ is of finite type over $k$ but $D_3(k)$ is infinite by assumption, we must have $\dim_k D_3>0.$ This implies the claim.

\it Second proof: \rm Note that the map 
$$\mathscr{A} \to \Res_{\Og_L/\Og_K} \mathscr{A}_{\Og_L}$$
is a closed immersion. The cokernel $\mathscr{Q}$ of this morphism (taken as an fppf-sheaf) is representable by a smooth group scheme of finite type over $\Og_K$ by \cite[Théorème 4.C]{An}. In particular, we obtain a canonical map $\mathscr{Q}\to\mathscr{C},$ which is generically an isomorphism. 
Since $\Og_K$ is strictly Henselian and the map $\Res_{\Og_L/\Og_K} \mathscr{A}_{\Og_L} \to \mathscr{Q}$ is smooth, the morphism 
$$\Res_{\Og_L/\Og_K} \mathscr{A}_{\Og_L}(\Og_K) \to \mathscr{Q}(\Og_K)$$
is surjective.  
By \cite[Corollary 2.3]{LLR}, the map $\mathscr{Q}\to\mathscr{C}$ is the composition of finitely many dilatations in smooth algebraic subgroups of the special fibre. Observe that at least one of the centres of these dilatations must have positive codimension. Indeed, the map $\mathscr{Q}\to\mathscr{C}$ would otherwise be an open immersion, so that the cokernel of $\Res_{L/K}A_L(K) \to C(K)$ would be finite, contradicting one of our assumptions. Using \cite[Proposition 2.2 (b)]{LLR}, we deduce that the $\Og_K$-length $\ell$ of the cokernel of the map
$$\Lie \mathscr{Q}\to\Lie \mathscr{C}$$ is positive. Once again, we denote by $\mathscr{G}_L^\bullet$ the sequence of Néron models induced by the base change of sequence (\ref{Rseq}) to $L,$ which is exact as we have already seen in the first proof. Let $q$ denote the $\Og_L$-length of the cokernel of the map
$$(\Lie \mathscr{Q})\otimes_{\Og_K}\Og_L \to \Lie \mathscr{G}_L ^ 3.$$ The snake lemma shows that $$c(\Res_{L/K} A_L) = \frac{q}{[L:K]} = c(C) + \ell,$$ which also implies the claim.
\end{proof}

Now we will construct a pair $(A, L)$
satisfying the conditions stated at the beginning of the section
in the equal characteristic case.

Assume that $K$ is the completed maximal unramified extension
of a complete discrete valuation field $K_{1}$ of equal characteristic
with finite residue field $k_{1}$.
Let $A$ be an ordinary elliptic curve over $K_{1}$
with good supersingular reduction.
Denote its base change to $K$ by the same symbol $A$.
Let $L_{1} / K_{1}$ be a totally ramified Galois extension of degree $p$.
Let $L = L_{1} K = L_1 \otimes_{K_1} K.$

\begin{proposition}
    If the valuation of the discriminant of $L / K$ is large enough
    (for a fixed $A$),
    then the kernel of the map
    $H^{1}(K, A) \to H^{1}(L, A_{L})$
    is infinite.
\end{proposition}

\begin{proof}
    For any integer $n \ge 1$,
    let $k_{n}$ be the degree $n$ subextension of $k / k_{1}$.
    Let $K_{n}$ and $L_{n}$ be the corresponding unramified extensions
    of $K_{1}$ and $L_{1}$, respectively.
    Then the kernel of the map
    $H^{1}(K_{n}, A) \to H^{1}(K, A)$
    is $H^{1}(\Gal(k / k_{n}), A(K))$,
    which is zero by \cite[Chapter I, Proposition 3.8]{Milne}
    (see also the erratum on Milne's homepage
    about the proof of this proposition)
    since $A$ has good reduction.
    Hence it is enough to show that the kernel of the map
    $H^{1}(K_{n}, A) \to H^{1}(L_{n}, A_{L})$
    has unbounded order as $n \to \infty$
    if the discriminant of $L / K$ is large enough.
    This kernel is isomorphic to
    $H^{1}(G, A(L_{n}))$,
    where $G = \Gal(L_{1} / K_{1})$.
    Denote the logarithm function to base $p$ by $\log_{p}$.
    
    We use the following result of Tan, Trihan and Tsoi \cite{TTT}:
    They construct,
    in \cite[Theorem A]{TTT},
    an explicit constant $C \in \Z$
    depending only on $A$, $L_{1} / K_{1}$ and not on $n$
    such that
        \[
                \log_{p} |H^{1}(G, A(L_{n}))|
            \ge
                C n
        \]
    for all $n \ge 1$.
    This constant $C$ is positive
    if the discriminant of $L / K$ is large enough
    by \cite[Theorem A (i)]{TTT}
    (the number $\flat_{w}$ written before \cite[Theorem A]{TTT} is $1$
    if the valuation $(p - 1) f_{v}$ of the discriminant of $L / K$ is
    large enough, so that \cite[Theorem A (i)]{TTT} applies).

    This immediately implies the result.
\end{proof}

\subsection{Non-additivity in general: Construction II} \label{consIIsubs}
We shall now give a construction of a second family of similar examples, making use of arithmetic duality theory as set out in Section \ref{Indrat-Proet-Section}; see \cite{Suz} for more details. As before, $\Og_K$ denotes a complete discrete valuation ring with algebraically closed residue field $k.$
Suppose we can find an étale isogeny $E' \to E$ of Abelian varieties over $K$ such that $c(E) \not= c(E').$ Let $\mathscr{F}_K$ be the Cartier dual of the kernel of this isogeny, which is multiplicative by assumption. Moreover, let $\mathscr{E}'$ and $\mathscr{E}$ be the Néron models of $E'$ and $E,$ respectively, so that we have an exact sequence of group schemes
$$0\to \mathscr{K} \to \mathscr{E}' \to \mathscr{E}$$ over $\Og_K,$ where $\mathscr{K}$ is the obvious kernel. Let $\mathscr{H}^\bullet: 0\to \mathscr{K} \to \mathscr{E}' \to \mathscr{E}\to 0$ be the associated complex, and let $D_{K,\mathscr{H}}$ be the $k$-group scheme constructed in Proposition \ref{Diconstructprop} such that 
$$H ^ 3 (\mathscr{H} ^ \bullet(\Og_K))=D_{K,\mathscr{H}}(k).$$
Finally, we embed $\mathscr{F}_K$ into an algebraic torus $T$ over $K$ and let $A$ and $A'$ be the dual Abelian varieties of $E$ and $E'$ respectively. If $T'$ denotes the quotient $T/\mathscr{F}_K,$ we obtain a commutative diagram
\begin{align}\begin{CD}
0@>>> \mathscr{F}_K @>>> T @>>> T' @>>> 0\\
&&@VVV@VVV@VV{\mathrm{Id}}V\\
0@>>> A @>>>  B @>>> T' @>>> 0 \label{Diag5}
\end{CD}\end{align}
with exact rows. Let 
$$\mathscr{G}^\bullet : \, 0\to \mathscr{A} \to \mathscr{B} \to \mathscr{T}' \to 0$$ be the complex of Néron lft-models induced by the bottom row of Diagram (\ref{Diag5}). Let $D_{K,\mathscr{G}}$ be the $k$-group scheme constructed in Proposition \ref{Diconstructprop} such that 
$$H^3(\mathscr{G}^{0,\bullet})=D_{K,\mathscr{G}}(k).$$
Let $L$ be a finite separable extension of $K$ with ring of integers $\Og_L$ such that $A,$ $B,$ and $T'$ all acquire semiabelian reduction over $L.$ Let $D_{L, \mathscr{G}}$ and $D_{L,\mathscr{H}}$ be the algebraic $k$-groups constructed analogously from the base changes of the relevant exact sequences of algebraic $K$-groups to $L.$
Just for the next Lemma, we drop the assumption that $c(E') \not=c(E).$
\begin{lemma}
We have $c(E')=c(E)$ if and only if $\dim_k D_{L, \mathscr{H}}=[L:K]\dim_k D_{K, \mathscr{H}}.$ \label{dimlem}
\end{lemma}
\begin{proof}
The Lemma can be proven exactly as Proposition \ref{Chigammaprop}, using that $\Lie \widetilde{\mathscr{K}}=0.$
\end{proof}

Now we reimpose the condition that $c(E')\not=c(E).$ Our next goal will be to compare the dimensions of $D_{K,\mathscr{G}}$ and $D_{K,\mathscr{H}};$ in fact, we shall see that they are equal. This part of the argument will require results from arithmetic duality. We shall work with the site $\Spec k^{\mathrm{indrat}}_{\mathrm{pro\acute{e}t}}$ introduced in Section \ref{Indrat-Proet-Section} above.
\begin{proposition}
There is a canonical morphism
$$D_{K,\mathscr{G}}^{\mathrm{perf}}\to\ker(\mathbf{H}^ 1(K,A)\to \mathbf{H}^ 1(K,A'))$$
of sheaves on $\Spec k^{\mathrm{indrat}}_{\mathrm{pro\acute{e}t}}$ with finite (étale) kernel and cokernel. In particular, $\ker(\mathbf{H}^ 1(K,A)\to \mathbf{H}^ 1(K,A')) \in \mathrm{Alg}/k.$\label{Dkerprop}
\end{proposition}
\begin{proof}
Let $\mathscr{D}$ be the presheaf $k' \mapsto \mathrm{coker} (\mathscr{B}^ 0 (\mathbf{O}_K(k')) \to \mathscr{T}'^0(\mathbf{O}_K(k')))$ on $\Spec k^{\mathrm{indrat}}_{\mathrm{pro\acute{e}t}}.$ By Remark \ref{kersmooth}, we obtain a canonical morphism $\mathscr{D} \to D_{K,\mathscr{G}}^{\mathrm{perf}},$ which becomes an isomorphism after sheafification. Now let $\Phi_{A'},\Phi_B, \Phi_{T'}$ be the groups of connected components of $\mathscr{A}'_k, \mathscr{B}_k,$ and $\mathscr{T}'_k,$ respectively, where $\mathscr{A}'$ is the Néron model of $A'.$ Since $A'$ is an Abelian variety, $\Phi_{A'}$ is finite. A simple diagram chase using \cite[Proposition 5.1.1.3 (2)]{HN2} shows that so are the kernel $\Psi_1$ and the cokernel $\Psi_2$ of $\Phi_B\to \Phi_{T'}.$ Using the snake lemma and the exactness of sheafification, we obtain an exact sequence 
$$\Psi_1 \to D_{K,\mathscr{G}}^{\mathrm{perf}} \to \ker(\mathbf{H}^ 1(K,A)\to \mathbf{H}^ 1(K,B)) \to \Psi_2$$ of sheaves on $\Spec k^{\mathrm{indrat}}_{\mathrm{pro\acute{e}t}}.$ Since $B$ also fits into a diagram analogous to (\ref{Diag1}), we have $\ker(\mathbf{H}^ 1(K,A)\to \mathbf{H}^ 1(K,B))=\ker(\mathbf{H}^ 1(K,A)\to \mathbf{H}^ 1(K,A'))$ and the Proposition follows.
\end{proof}
\begin{lemma}
There exists a morphism
$$\mathbf{Ext}^1_{k^{\mathrm{indrat}}_{\mathrm{pro\acute{e}t}}}(D_{K,\mathscr{H}}^{\mathrm{perf}}, \Q/\Z) \to \ker(\mathbf{H}^ 1(K,A)\to \mathbf{H}^ 1(K,A'))$$
which has finite kernel and cokernel. \label{Lem47}
\end{lemma}
\begin{proof}
From the exact sequence $0\to \mathscr{K}_K \to E' \to E \to 0$ we obtain an exact sequence
\begin{align*}
0\to \boldsymbol{\Gamma}(K, \mathscr{K}_K)\to \boldsymbol{\Gamma}(K, E')\overset{\delta}{\to} \boldsymbol{\Gamma}(K, E) \to D_{K,\mathscr{H}}^{\mathrm{perf}} \to 0
\end{align*}
in $\mathrm{PAlg} / k.$ Denote by $\Delta$ the image of $\delta.$ Then we have exact sequences
\begin{align*}
\mathbf{Hom}_{k^{\mathrm{indrat}}_{\mathrm{pro\acute{e}t}}}(\pi_0(\Delta), \Q/\Z) &\to \mathbf{Ext}^1_{k^{\mathrm{indrat}}_{\mathrm{pro\acute{e}t}}}(D_{K,\mathscr{H}}^{\mathrm{perf}}, \Q/\Z)\\ &\to
\mathbf{Ext}^1_{k^{\mathrm{indrat}}_{\mathrm{pro\acute{e}t}}}(\boldsymbol{\Gamma}(K, E), \Q/\Z)
\to \mathbf{Ext}^1_{k^{\mathrm{indrat}}_{\mathrm{pro\acute{e}t}}}(\Delta, \Q/\Z) \to 0
\end{align*}
and 
\begin{align*}
\mathbf{Hom}_{k^{\mathrm{indrat}}_{\mathrm{pro\acute{e}t}}}(\boldsymbol{\Gamma}(K, \mathscr{K}_K), \Q/\Z) &\to \mathbf{Ext}^1_{k^{\mathrm{indrat}}_{\mathrm{pro\acute{e}t}}}(\Delta, \Q/\Z) \to \mathbf{Ext}^1_{k^{\mathrm{indrat}}_{\mathrm{pro\acute{e}t}}}(\boldsymbol{\Gamma}(K, E'), \Q/\Z).
\end{align*}
Here we use \cite[ Proposition 2.4.1(a)]{Suz}. Moreover, by Proposition \ref{SuzDualityProp}, we have a commutative diagram
$$\begin{CD}
\mathbf{Ext}^1_{k^{\mathrm{indrat}}_{\mathrm{pro\acute{e}t}}}(\boldsymbol{\Gamma}(K, E), \Q/\Z)@>>>\mathbf{Ext}^1_{k^{\mathrm{indrat}}_{\mathrm{pro\acute{e}t}}}(\boldsymbol{\Gamma}(K, E'), \Q/\Z)\\
@AAA@AAA\\
\mathbf{H}^ 1(K,A) @>>> \mathbf{H}^ 1(K,A')
\end{CD}$$
with vertical isomorphisms. Hence all we must show is that the canonical map from $\mathbf{Ext}^1_{k^{\mathrm{indrat}}_{\mathrm{pro\acute{e}t}}}(D_{K,\mathscr{H}}^{\mathrm{perf}}, \Q/\Z)$ into the kernel of the top row of the diagram has finite kernel and cokernel. However, this follows from the fact that $\pi_0(\Delta)$ and $\boldsymbol{\Gamma}(K, \mathscr{K}_K)$ are finite (see Proposition \ref{SuzFinFlatProp} for the second of those objects). 
\end{proof}
\begin{corollary}
We have $\dim_k D_{K,\mathscr{G}}=\dim_k D_{K,\mathscr{H}}.$ \label{dimcor}
\end{corollary}
\begin{proof}
By construction, $D_{K,\mathscr{G}}$ and $D_{K,\mathscr{H}}$ are smooth algebraic groups over $k$ such that $D_{K,\mathscr{H}}(k)$ embeds into $H ^ 1 (K, \mathscr{K}_K)$ and such that there is a homomorphism $D_{K,\mathscr{G}}(k) \to H ^ 1(K,A)$ with finite kernel (the finiteness is a consequence of the snake lemma and the fact that $\mathscr{A}$ is of finite type over $\Og_K$). We may here assume that $k$ is uncountable by replacing $\Og_K$ with an extension of ramification index 1, such as $\Og_K \subseteq \mathbf{O}_K(k')$ with an uncountable algebraically closed extension $k'$ of $k.$ Since both (étale) cohomology groups just mentioned are torsion, our assumption on $k$ guarantees that both $D_{K,\mathscr{G}}$ and $D_{K,\mathscr{H}}$ have unipotent identity components. Therefore the Corollary follows from Proposition \ref{Dkerprop} and Lemma \ref{Lem47} together with the fact that the dimension is invariant under taking perfections as well as isogenies and Breen-Serre duals (see \cite[Proposition 1.2.1]{Beg} for the last claim).
\end{proof}

We are now ready to prove
\begin{theorem}
Let $0 \to A \to B\to T' \to 0$ be the bottom sequence from diagram (\ref{Diag5}), and assume that $c(E') \not=c(E).$ Then $c(B) \not = c(A) + c(T').$
\end{theorem}
\begin{proof}
Let $L$ be a finite separable extension of $K$ over which all semiabelian varieties in the sequence from the Theorem acquire semiabelian reduction. Corollary \ref{dimcor} and Lemma \ref{dimlem} together imply that $\dim_k D_{L, \mathscr{G}} \not = [L:K]\dim_k D_{K, \mathscr{G}},$ so the claim follows from Proposition \ref{Chigammaprop}. 
\end{proof}

In particular, we obtain an example of non-additivity of $c(-)$ as soon as we can find an étale isogeny $E'\to E$ such that $c(E')\not=c(E).$ Such an example is constructed in \cite[(6.10.1)]{Chai} over the base field $\Q_3.$
Note that this field is of characteristic 0 and has finite residue field, so it satisfies both parts of Chai's \it awkward assumption \rm (\it op. cit., \rm p. 733). 

\subsection{Examples with imperfect residue field} \label{ImperfectResPara}
In this paragraph, we shall construct examples which show that the base change conductor is not additive in general in exact sequences $0\to T\to B\to A\to 0$ of semiabelian varieties if the residue field of $K$ is imperfect, even when $T$ is a torus. We shall begin by constructing such an example where $T,$ $B,$ and $A$ are all tori, which moreover shows that the base change conductor is not invariant under isogeny if the residue field is imperfect. Using rigid geometry, we shall use this exact sequence of tori to give a counterexample to the statement of Chai's conjecture in this setup.
\subsubsection{Algebraic tori}
Let $\Og_K$ be a complete discrete valuation ring with separably closed \it imperfect \rm residue field $\kappa$ of equal characteristic 2. Let $\pi$ be a uniformiser of $\Og_K$ and $c\in \Og_K$ be an element whose image in $\kappa$ is not contained in $\kappa^2.$ Consider the polynomial 
$$f_i(X):=X^2 + \pi ^ i X + c\in \Og_K[X]$$ for some $i \in \N.$ The polynomial $f_i$ is clearly irreducible and separable over $K.$ 
\begin{lemma}
Let $K_i$ be the splitting field of $f_i$ and let $T_i$ be the torus $(\Res_{K_i/K}\Gm)/\Gm.$ \label{piilem} Then 
$$c(T_i)=i.$$
\end{lemma}
\begin{proof}
Because the sequence of Néron lft-models induced by the exact sequence $0\to \Gm \to \Res_{K_i/K}\Gm \to T_i \to 0$ is exact \cite[Lemma 11.2]{CY}, we have $c(T_i)=c(\Res_{K_i/K}\Gm).$ One checks easily that the base change conductor of the induced torus is equal to one-half of the valuation of the discriminant of $K_i,$ so the claim is proven.
\end{proof}

Now let $L$ be the compositum of $K_i$ and $K_j$ for distinct $i,j \in \N.$ Then 
$$\Gal(L/K)\cong \Z/2\Z \times \Z/2\Z$$
(non-canonically). We fix such an isomorphism by choosing generators $\sigma_i, \sigma_j$ of $\Gal(L/K)$ such that $\sigma_i$ fixes $K_j$ and $\sigma_j$ fixes $K_i.$ Let $T$ be the torus over $K$ 
whose character lattice is given by $\Z^2$ with the Galois action 
$$\sigma_i \mapsto \begin{pmatrix} -1 & 1 \\ 0 & 1 \end{pmatrix}, \hspace{0.5 in} \sigma_j \mapsto \begin{pmatrix} 1 & -1 \\ 0 & -1 \end{pmatrix}.$$
By construction, we have an exact sequence
$$0 \to T_j\to T  \to T_i \to 0$$
of algebraic tori. From now on, we shall put $i:=1$ and $j:=2.$ This will simplify later calculations. 

Note that we already know the base change conductors of $T_1$ and $T_2.$ In order to calculate that of $T,$ we choose roots $\alpha_i$ of $f_i$ for $i=1,2,$ and put 
$$\beta:= \pi \alpha_1 + \alpha_2 \in L.$$
A simple calculation shows that
$$g(X):=X^2 + \pi^2X + \pi ^2 c + c $$
is the minimal polynomial of $\beta.$ Let $F:=K(\beta).$ We claim that the kernel of the canonical projection $\Z[\Gal(L/K)] \to \Z[\Gal(F/K)]\to 0$ is isomorphic to $X^\ast(T).$ Clearly, the elements $\mathrm{Id}_L - \sigma_1\sigma_2$ and $\sigma_1-\sigma_2$ lie in the kernel of this map; this follows from the fact that, by Galois theory, $F$ is the maximal subfield of $L$ on which $\sigma_1\sigma_2$ acts trivially. Moreover, these two elements generate a saturated sublattice of $\Z[\Gal(L/K)]$ of rank 2, which must therefore be the whole kernel. Putting $e_1:=1-\sigma_1+\sigma_2-\sigma_1\sigma_2$ and $e_2:=\sigma_1-\sigma_2,$ we see that this kernel and $X^\ast(T)$ are indeed isomorphic as integral $\Gal(L/K)$-representations. We have, in particular, \it resolved $T$ by induced tori \rm by constructing an exact sequence
$$0\to \Res_{F/K}\Gm \to \Res_{L/K}\Gm \to T \to 0.$$
\begin{lemma}
We have $c(\Res_{L/K}\Gm)=6.$ \label{matrixlem}
\end{lemma}
\begin{proof}
Let us first show that $\Og_L=\Og_K[\alpha_1, \alpha_2].$ To see this, put $\gamma:=\alpha_1+\alpha_2.$ Then $\gamma$ is a root of the Eisenstein polynomial $X^2 + \pi^2 X + \pi \alpha_1+\pi^2\alpha_1$ over $K(\alpha_1).$ Because $f_1$ is irreducible after reduction modulo $\pi,$ we have $\Og_{K_1}=\Og_K[\alpha_1],$ so we obtain 
$$\Og_L=\Og_K[\alpha_1][\gamma]=\Og_K[\alpha_1, \alpha_2].$$ Put $\delta:=\alpha_1\alpha_2.$ The number $2c(T)=e_{L/K}c(T)$ is therefore equal to the $\Og_L$-length of the cokernel of the map
$$\Og_L\otimes_{\Og_K} \Og_L \to \Og_L^4$$
given by the matrix
$$\begin{pmatrix}
1 & \alpha_1 & \alpha_2 & \delta \\
1 & \alpha_1 + \pi & \alpha_2 & \delta + \pi \alpha_2 \\
1 & \alpha_1 & \alpha_2 + \pi^2 & \delta + \pi^2\alpha_1 \\
1 & \alpha_1 + \pi & \alpha_2 + \pi ^2 & \delta + \pi ^2 \alpha_1 + \pi\alpha_2 + \pi^3
\end{pmatrix}$$
with respect to the basis $1\otimes 1, \alpha_1\otimes 1, \alpha_2 \otimes 1, \delta \otimes 1$ on the source and the standard basis on the target. Subtracting the first row from the other three and then the second and third from the fourth transforms this matrix into 
$$ \begin{pmatrix}
1 & \alpha_1 & \alpha_2 & \delta \\
0 & \pi & 0 & \pi \alpha_2 \\
0 & 0 & \pi ^2 & \pi^2 \alpha_1 \\
0 & 0 & 0 & \pi^3
\end{pmatrix}.$$
This shows that the cokernel has a composition series with successive quotients $\Og_L/\langle \pi \rangle,$ $\Og_L/\langle \pi^2 \rangle,$ and $\Og_L/\langle \pi^3 \rangle.$ The $\Og_L$-lengths of these modules are 2, 4, and 6, respectively, so that the length of the cokernel is indeed equal to 6.
\end{proof}
\begin{corollary}
Let $0\to T_2 \to T \to T_1 \to 0$ be the sequence of algebraic $K$-tori constructed above. \label{nonaddcor}Then $c(T) \not=c(T_2)+ c(T_1).$
\end{corollary}
\begin{proof}
We already know from Lemma \ref{piilem} that $c(T_1)=1$ and $c(T_2)=2.$ Moreover, the proof of the same Lemma shows that $c(\Res_{F/K}\Gm)=2.$ Hence the resolution 
$$0\to \Res_{F/K}\Gm \to \Res_{L/K}\Gm \to T \to 0,$$ together with \cite[Remark 4.8 (a)]{Chai} and Lemma \ref{matrixlem} above shows that $c(T)= 6-2=4.$
\end{proof}

\begin{remark} Since $T$ is isogenous to $T_2 \times_K T_1,$ the Corollary also shows that the base change conductor of an algebraic torus is not invariant under isogeny if the residue field is imperfect. This shows in particular that the formula given by Chai, Yu, and de Shalit for the base change conductor of a torus $T$ in terms of the Galois module $X_\ast(T)\otimes_{\Z}\Q$ \cite[Theorem on p. 367 and Theorem 12.1]{CY} cannot be generalised to the case of imperfect residue fields.
\end{remark}
\subsubsection{Raynaud extensions}
We shall now construct an exact sequence $0\to T_2 \to B \to E\to 0$ such that $T_2$ is the torus constructed in the preceding paragraph, $E$ is an elliptic curve over $K,$ and such that $c(B)\not=c(T_2) + c(E).$ 
Recall that the character lattice of an algebraic torus can be seen as an étale group scheme over $K$ which becomes isomorphic to a constant group scheme $\Z^d$ for some $d\in \N$ after a finite separable extension of $K.$ Such a group scheme will be called an \it étale lattice \rm over $K.$ We have a canonical isomorphism 
$$X^\ast(T_1) \otimes_{\Z} X^\ast(T_1) = \Z,$$
where $\Z$ carries the trivial Galois action. Note that we also have an exact sequence 
$$0 \to T_1 \to T \to T_2 \to 0;$$ this can be seen directly from the definition of the Galois action on $X^\ast(T).$ In particular, we obtain a morphism
$$X^\ast(T) \otimes_{\Z} X^\ast(T_1) \to X^\ast(T_1) \otimes_{\Z} X^\ast(T_1) = \Z \to \Gm, $$
where the last map is given by $1\mapsto q $ for some $q\in K^\times$ with $v_K(q)>0.$ This map induces a commutative diagram
$$\begin{CD}
X^\ast(T_1) @>{=}>> X ^\ast(T_1)\\
@VVV@VVV\\
T @>>> T_1.
\end{CD}$$
It is easy to see that the quotient $T_1/X^\ast(T_1)$ (taken as a rigid $K$-group) is equal to the rigid $K$-group associated with the quadratic twist $E$ of the Tate curve with period $q$ along the extension $K\subseteq K_1.$ To proceed, we shall need the following
\begin{lemma}
Let $0\to \mathscr{F} \to \mathscr{G} \to \mathscr{H} \to \mathscr{E}$
be a complex of smooth formal group schemes over $\mathrm{Spf} \, \Og_K$ which is exact as a complex of sheaves on the small formal smooth site of $\Og_K,$ and such that $\mathscr{E}$ is étale. Then the map $\mathscr{G}\to \mathscr{H}$ is smooth. \label{issmoothlem}
\end{lemma}
\begin{proof}
We may replace $\mathscr{H}$ by an open formal subgroup scheme without changing the conclusion of the Lemma. But since $\mathscr{E}$ is étale, its unit section is an open immersion, so we may assume without loss of generality that $\mathscr{E}=0.$ Then our assumption implies that the map $\mathscr{G}\to \mathscr{H}$ has sections locally in the smooth topology. Denoting by $\mathscr{G}_n$ and $\mathscr{H}_n$ the levels of $\mathscr{G}$ and $\mathscr{H},$ respectively, for $n\in \N,$ this implies that the maps $\Lie \G_n \to \Lie \mathscr{H}_n$ are surjective for all such $n.$ Hence the claim follows.
\end{proof}

Now consider the exact sequence
$$0 \to T_2 \to T/X ^ \ast(T_1) \to E\to 0 $$
of sheaves on the small rigid smooth site induced by the sequence $0\to T_2 \to T \to T_1 \to 0.$
Then we have
\begin{proposition}
The rigid $K$-group $T/X ^ \ast(T_1)$ is the rigid $K$-group associated with a semiabelian variety $B$ over $K,$ and we have an exact sequence $$0 \to T_2 \to B \to E \to 0$$ of semiabelian varieties over $K.$ Moreover, we have $c(B)=c(T)$ and $c(E)=c(T_1).$ In particular, $c(B) \not= c(T_2) + c(E).$
\end{proposition}
\begin{proof}
The first two claims will be proven together. Note that, by Galois descent, we may replace $K$ by a finite Galois extension $L$ which splits $T_2$ and $X^\ast(T_1).$ It is well-known that there is a canonical isomorphism 
$E_L(L)=\mathrm{Ext}^1_L (E_L, \Gm)$ in the category of commutative $L$-group schemes. By the discussion on p. 1272 of \cite{Bosch}, the analogous isomorphism exists with respect to the rigid category. This shows that the exact sequence 
$$0 \to T_2 \to T/X^\ast(T_1) \to T_1 / X^\ast (T_1) \to 0$$ algebraises uniquely. In order to see the second claim, we consider the exact sequence
$$0 \to X^\ast(T_1) \to T \to B \to 0$$ of rigid analytic $K$-groups. By \cite[Proposition 4.5]{BX}, we obtain an induced complex
$$0 \to \widehat{\mathscr{N}} \to \widehat{\mathscr{T}} \to \widehat{\mathscr{B}} \to H ^ 1 (\Gal(K\sep/K), X ^\ast(T_1))$$ of smooth formal algebraic groups over $\Spf \Og_K$ which is exact when restricted to the small formal smooth site of $\Spf \Og_K.$ Here, the first four objects are the formal Néron lft-models of $0,$ $X^\ast(T_1),$ $T,$ and $B,$ respectively \cite[Proposition 4.5]{BX}. The notation is justified by the fact that those formal group schemes are the formal completions of the usual Néron lft-models  \cite[Theorem 6.2]{BS}. Because $\widehat{\mathscr{N}}$ is étale over $\Spf \Og_K,$ Lemma \ref{issmoothlem} tells us that the maps $$\Lie \mathscr{T} \otimes_{\Og_K} \Og_K/\mathfrak{m}_K^n \to \Lie \mathscr{B} \otimes_{\Og_K} \Og_K/\mathfrak{m}_K^n$$ are isomorphisms for all $n\in \N.$ Hence the map $\Lie \mathscr{T}\to \Lie \mathscr{B}$ is an isomorphism. Note, moreover, that this isomorphism is compatible with finite separable extensions of $K.$ This shows that $c(B)=c(T);$ the proof that $c(E)=c(T_1)$ is entirely analogous (alternatively, one could invoke \cite[Proposition 5.1]{Chai}). The final claim is now an immediate consequence of Corollary \ref{nonaddcor}.
\end{proof}


\section{Invariance under duality and isogeny}

In this section, we shall prove that the base change conductor is invariant under duality for Abelian varieties (Theorem \ref{0023}), answering a question of Chai \cite[8.2]{Chai}.
We shall also prove that the base change conductor is invariant under isogeny for tori (Theorem \ref{0025}), giving another proof for a result of Chai, Yu, and de Shalit \cite[Theorem on p. 367 and Theorem 12.1]{CY}.

\subsection{Dimension counting functions}
Let $D^{b}(\mathrm{IPAlg} / k)$ be
the bounded derived category of $\mathrm{IPAlg} / k$.

\begin{definition}
	Define full subcategories $\mathcal{D}$
	(resp.\ $\mathcal{D}_{u}$) of $D^{b}(\mathrm{IPAlg} / k)$
	to be the smallest full triangulated subcategory
	closed under isomorphisms
	containing perfections of connected (resp.\ connected unipotent) algebraic groups,
	pro-finite-\'etale groups and torsion \'etale groups.
\end{definition}

Let $K_{0}(\mathcal{D})$ be the Grothendieck group of $\mathcal{D}$.

\begin{proposition}
	There exists a unique group homomorphism
	$\chi \colon K_{0}(\mathcal{D}) \to \Z$
	that sends the perfection of a connected algebraic group to its dimension
	and kills pro-finite-\'etale groups and torsion \'etale groups.
\end{proposition}

\begin{proof}
	The uniqueness is clear.
	For the existence,
	recall the universal covering functor
	$\mathrm{PAlg} / k \to \mathrm{PAlg} / k$, $G \mapsto \overline{G}$
	from \cite[Section 6.2, Definition 2]{Ser60},
	which is exact by \cite[Section 10.3, Theorem 3]{Ser60}.
	It preserves the dimension of $G \in \mathrm{Alg} / k$
	and kills pro-finite-\'etale groups.
	It extends to an exact endofunctor $G \mapsto \overline{G}$ on $\mathrm{IPAlg} / k$
	and hence to a triangulated endofunctor on $D^{b}(\mathrm{IPAlg} / k)$,
	which kills torsion \'etale groups.
        Combining all these facts, it formally follows that for
	any $G \in \mathcal{D}$ and any $n \in \Z$,
	the object $H^{n}(\overline{G})$ is the universal covering of a connected algebraic group.
	Therefore we have a well-defined integer
	$\sum_{n} (-1)^{n} \dim H^{n}(\overline{G})$ for each $G \in \mathcal{D}$.
	This assignment factors through $K_{0}(\mathcal{D})$.
\end{proof}

For any morphism $f$ in $D^{b}(\mathrm{IPAlg} / k)$
whose mapping cone belongs to $\mathcal{D}$,
we define $\chi(f)$ to be the value of $\chi$ at a mapping cone of $f$.

Let $D(k^{\mathrm{indrat}}_{\mathrm{proet}})$ be the derived category
of sheaves of abelian groups on $\Spec k^{\mathrm{indrat}}_{\mathrm{proet}}$.
The fully faithful embedding
    $
            \mathrm{IPAlg} / k
        \hookrightarrow
            \mathrm{Ab}(k^{\mathrm{indrat}}_{\mathrm{proet}})
    $
induces a fully faithful embedding
    $
            D^{b}(\mathrm{IPAlg} / k)
        \hookrightarrow
            D(k^{\mathrm{indrat}}_{\mathrm{proet}})
    $
again by \cite[Proposition (2.3.4)]{Suz}.
Let $G \mapsto G^{\mathrm{SD}}$ be the contravariant endofunctor
	$
		R \operatorname{\mathbf{Hom}}
		_{k^{\mathrm{indrat}}_{\mathrm{proet}}}
		(\;\cdot\;, \Z)
	$
on $D(k^{\mathrm{indrat}}_{\mathrm{proet}})$.

\begin{proposition} \label{0016} \mbox{}
	\begin{enumerate}
		\item \label{0000}
			The functor $\mathrm{SD}$ maps $\mathcal{D}$ to $\mathcal{D}_{u}$
			and restricts to a contravariant autoequivalence on $\mathcal{D}_{u}$.
		\item \label{0001}
			For any $G \in \mathcal{D}_{u}$, we have $\chi(G) = \chi(G^{\mathrm{SD}})$.
	\end{enumerate}
\end{proposition}

\begin{proof}
	\eqref{0000} is \cite[Proposition (2.4.1) (b), (d)]{Suz}.
	\eqref{0001} is \cite[Proposition 1.2.1]{Beg}
	(note that
            $
                    (\mathbf{G}_{\mathrm{a}}^{\mathrm{perf}})^{\mathrm{SD}}
                \cong
                    (\mathbf{G}_{\mathrm{a}}^{\mathrm{perf}})[-2]
            $
        by the paragraph after the proof of
        \cite[Proposition (2.4.1)]{Suz}).
\end{proof}


\subsection{Base change conductors as dimensions}

\begin{definition}
	Let $G$ be a semiabelian variety over $K$ with N\'eron lft-model $\mathscr{G}$.
	Let $L / K$ be a finite extension.
	Define $c(G)_{L}'$ to be the dimension of 
	$(\mathbf{\Gamma}(L, G) / \mathbf{\Gamma}(\mathcal{O}_{L}, \mathscr{G}))^{0}$.
\end{definition}

This definition makes sense
by the first statement of the following proposition:

\begin{proposition} \label{0015}
	Let $G$ be a semiabelian variety over $K$ with N\'eron lft-model $\mathscr{G}$.
	Let $L / K$ be a finite extension.
	Then $(\mathbf{\Gamma}(L, G) / \mathbf{\Gamma}(\mathcal{O}_{L}, \mathscr{G}))^{0}$
	is an object of $\mathrm{Alg} / k$.
	If $G$ has semiabelian reduction over $L$,
	then its dimension $c(G)_{L}'$ is equal to the base change conductor $c(G)$ times $e_{L / K}$.
\end{proposition}

\begin{proof}
	Let $\mathscr{G}_{L}$ be the N\'eron lft-model of $G \times_{K} L$
	($= G \times_{\Spec K} \Spec L$).
	Let $\mathscr{G}'$ be the kernel of the natural morphism
	$\mathscr{G}^{0} \times_{\mathcal{O}_{K}} \mathcal{O}_{L} \to \mathscr{G}_{L}^{0}$.
	Then the smooth algebraic group associated with the sequence
		\begin{equation} \label{0026}
				0
			\to
				\mathscr{G}'
			\to
				\mathscr{G}^{0} \times_{\mathcal{O}_{K}} \mathcal{O}_{L}
			\to
				\mathscr{G}_{L}^{0}
			\to
				0
		\end{equation}
	by Proposition \ref{Diconstructprop} is
		$
				\operatorname{Gr}(\mathscr{G}_{L}^{0})
			/
				\operatorname{Gr}(
						\mathscr{G}^{0}
					\times_{\mathcal{O}_{K}}
						\mathcal{O}_{L}
				)
		$.
	Its perfection is
		\[
					\mathbf{\Gamma}(\mathcal{O}_{L}, \mathscr{G}_{L}^{0})
				/
					\mathbf{\Gamma}(\mathcal{O}_{L}, \mathscr{G}^{0})
			\cong
					\mathbf{\Gamma}(L, G)^{0}
				/
					\mathbf{\Gamma}(\mathcal{O}_{L}, \mathscr{G})^{0}.
		\]
	If $G$ has semiabelian reduction over $L$,
	then its dimension is the value of $\chi_{\Lie}$ at the sequence \eqref{0026}
        viewed as a complex with non-zero terms in degrees $0$, $1$ and $2$
        by Theorem \ref{Liepointsthm}.
        But $\widetilde{\mathscr{G}'} = 0$
        and the map $\Lie(\mathscr{G}^{0} \times_{\mathcal{O}_{K}} \mathcal{O}_{L}) \to \Lie(\mathscr{G}^{0}_{L})$ is injective
        with torsion cokernel of length $e_{L / K} c(G)$.
        Hence the value of $\chi_{\Lie}$ at \eqref{0026} is $e_{L / K} c(G)$.
	Since the natural morphism
		\[
					\mathbf{\Gamma}(L, G)^{0}
				/
					\mathbf{\Gamma}(\mathcal{O}_{L}, \mathscr{G})^{0}
			\to
				(\mathbf{\Gamma}(L, G) / \mathbf{\Gamma}(\mathcal{O}_{L}, \mathscr{G}))^{0}
		\]
	is surjective with finite kernel, the result follows.
\end{proof}

\begin{proposition} \label{IsogL}
Let $\varphi \colon G_1\to G_2$ be an isogeny of semiabelian varieties over $K$ and let $L/K$ be a finite Galois extension. For $i=1,2,$ let $\mathscr{G}_i$ be the Néron lft-model of $G_i$ over $\Og_K.$ Then the morphism 
$$(\mathbf{\Gamma}(L,G_{1,L})/\mathbf{\Gamma}(\Og_L,\mathscr{G}_{1,L})) ^ 0\to (\mathbf{\Gamma}(L,G_{2,L})/\mathbf{\Gamma}(\Og_L,\mathscr{G}_{2,L})) ^ 0 $$
induces an isogeny on maximal semiabelian quotients. In particular, the mapping cone of this morphism (in $\mathcal{D}$) lies in $\mathcal{D}_u.$
\end{proposition}
\begin{proof}
Let $Q_1$ and $Q_2$ be the maximal semiabelian quotients of the left and right hand side, respectively. Let $d$ be the degree of the isogeny $\varphi$, so we have another isogeny $\psi\colon G_2 \to G_1$ such that $\psi\circ\varphi$ and $\varphi\circ \psi$ are both multiplication by $d.$ Then the compositions $Q_1 \to Q_2 \to Q_1$ and $Q_2 \to Q_1 \to Q_2$ are multiplication by $d,$ which is an isogeny as the $Q_j$ are semiabelian. Hence both maps $Q_1\to Q_2$ and $Q_2 \to Q_1$ are surjective and have finite kernel, as was claimed. In particular, both the kernel and the cokernel of the morphism from the proposition are extensions of a connected unipotent perfect algebraic group and a finite étale (perfect) algebraic group.
\end{proof}


\subsection{Pairings on N\'eron models}

Let $A$ and $B$ be Abelian varieties over $K$ dual to each other.
Let $\mathscr{A}$ and $\mathscr{B}$ be their N\'eron models.
Let $\otimes^{L} = \otimes^{L}_{\Z}$ be the derived tensor product functor.
As explained in \cite[Sections 4 and 5]{Bosch} or \cite[Chapter III, Appendix C]{Milne},
the morphism $A \otimes^{L} B \to \Gm[1]$ in $D(K_{\fppf})$
($= D(\Spec K_{\fppf})$) defined by the Poincar\'e biextension
canonically extends to a morphism
	\begin{equation} \label{0006}
		\mathscr{A} \otimes^{L} \mathscr{B}^{0} \to \Gm[1]
	\end{equation}
in $D(\mathcal{O}_{K, \fppf})$.

For a morphism $X \to Y$ in an abelian category,
we denote by $[X \to Y]$ the complex
concentrated in degrees $-1$ and $0$.
Its image in the derived category is
also denoted by $[X \to Y]$.

\begin{proposition} \label{BiextK}
	Let $u \colon A_{1} \to A_{2}$ be a morphism of Abelian varieties over $K$.
	Let $v \colon B_{2} \to B_{1}$ be the morphism induced on the duals.
	Then there exists a canonical morphism
		\begin{equation} \label{0007}
					[A_{1} \to A_{2}][-1]
				\otimes^{L}
					[B_{2} \to B_{1}]
			\to
				\Gm[1]
		\end{equation}
	in $D(K_{\fppf})$ such that the diagram
		\[
			\begin{CD}
					[A_{1} \to A_{2}][-1]
				@>>>
					A_{1}
				@>>>
					A_{2}
				\\
				@VVV @VVV @VVV
				\\
					R \operatorname{\mathbf{Hom}}_{K_{\fppf}}([B_{2} \to B_{1}], \Gm[1])
				@>>>
					R \operatorname{\mathbf{Hom}}_{K_{\fppf}}(B_{1}, \Gm[1])
				@>>>
					R \operatorname{\mathbf{Hom}}_{K_{\fppf}}(B_{2}, \Gm[1])
			\end{CD}
		\]
	is a morphism of exact triangles,
	where the middle and right vertical morphisms are
	the morphisms induced by the Poincar\'e biextensions.
\end{proposition}

\begin{proof}
	For $i = 1, 2$,
	let $P_{i}$ be the (rigidified) Poincar\'e bundle on $A_{i} \times_{K} B_{i}$.
	Then the pullback of $P_{1}$ by the morphism
	$\mathrm{id} \times v \colon A_{1} \times B_{2} \to A_{1} \times B_{1}$
	and the pullback of $P_{2}$ by the morphism
	$u \times \mathrm{id} \colon A_{1} \times B_{2} \to A_{2} \times B_{2}$
	are canonically identified.
	Then the explicit construction of the morphisms $A_{i} \otimes^{L} B_{i} \to \Gm[1]$
	from $P_{i}$ in \cite[Expos\'e VII, Theorems 3.2.5 and 3.6.4]{Gro72}
	gives a canonical choice of a left vertical morphism in the diagram
	that completes it into a morphism of exact triangles.
\end{proof}

\begin{proposition} \label{0011}
	Let $u \colon A_{1} \to A_{2}$ be a morphism of Abelian varieties over $K$.
	Let $v \colon B_{2} \to B_{1}$ be the morphism induced on the duals.
	Let $u \colon \mathscr{A}_{1} \to \mathscr{A}_{2}$
	and $v \colon \mathscr{B}_{2} \to \mathscr{B}_{1}$ be
	the morphisms induced on the N\'eron models.
	Then the morphism \eqref{0007} canonically extends to a morphism
		\begin{equation} \label{0012}
					[\mathscr{A}_{1} \to \mathscr{A}_{2}][-1]
				\otimes^{L}
					[\mathscr{B}_{2}^{0} \to \mathscr{B}_{1}^{0}]
			\to
				\Gm[1]
		\end{equation}
	in $D(\mathcal{O}_{K, \fppf})$.
	The resulting diagram
		\begin{equation} \label{0013}
			\begin{CD}
					[\mathscr{A}_{1} \to \mathscr{A}_{2}][-1]
				@>>>
					\mathscr{A}_{1}
				@>>>
					\mathscr{A}_{2}
				\\
				@VVV @VVV @VVV
				\\
					R \operatorname{\mathbf{Hom}}
					_{\mathcal{O}_{K, \fppf}}
					(
                            [\mathscr{B}_{2}^{0} \to \mathscr{B}_{1}^{0}], \Gm[1]
                        )
				@>>>
					R \operatorname{\mathbf{Hom}}
					_{\mathcal{O}_{K, \fppf}}
					(\mathscr{B}_{1}^{0}, \Gm[1])
				@>>>
					R \operatorname{\mathbf{Hom}}
					_{\mathcal{O}_{K, \fppf}}
					(\mathscr{B}_{2}^{0}, \Gm[1])
			\end{CD}
		\end{equation}
is a morphism of exact triangles.
\end{proposition}

\begin{proof}
        Let $\mathscr{P}_{i}$ be the canonical extension of
        the Poincar\'e $\Gm$-biextension on $A_{i} \times B_{i}$ to
        $\mathscr{A}_{i} \times \mathscr{B}_{i}^{0}$.
   	Then the pullback of $\mathscr{P}_{1}$ by the morphism
	   $
                \mathrm{id} \times v
            \colon
                \mathscr{A}_{1} \times \mathscr{B}_{2}^{0}
            \to
                \mathscr{A}_{1} \times \mathscr{B}_{1}^{0}
            $
	and the pullback of $\mathscr{P}_{2}$ by the morphism
	   $
                u \times \mathrm{id}
            \colon
                \mathscr{A}_{1} \times \mathscr{B}_{2}^{0}
            \to
                \mathscr{A}_{2} \times \mathscr{B}_{2}^{0}
          $
	give two $\Gm$-biextensions on
        $\mathscr{A}_{1} \times \mathscr{B}_{2}^{0}$.
        Their generic fibers are canonically identified
        by the proof of Proposition \ref{BiextK}.
        Therefore, by the fully faithfulness result of biextensions
        in \cite[Expos\'e VIII, Theorem 7.1 (b)]{Gro72},
        we know that these two $\Gm$-biextensions on
        $\mathscr{A}_{1} \times \mathscr{B}_{2}^{0}$
        are canonically identified.
        By the same argument as the proof of Proposition \ref{BiextK},
        this gives a canonical choice of a left vertical morphism
        in the diagram that completes it into a morphism of exact triangles.
\end{proof}


\subsection{Dimensions of duality pairings}

Recall from \cite[(5.2.1.1)]{Suz} the canonical ``trace'' isomorphism
	\begin{equation} \label{0008}
			R \mathbf{\Gamma}_{x}(\mathcal{O}_{K}, \Gm)
		\cong
			\Z[-1]
	\end{equation}
in $D(k^{\mathrm{indrat}}_{\mathrm{proet}})$,
where $\mathbf{\Gamma}_{x}(\mathcal{O}_{K}, \;\cdot\;)$ denotes
the kernel of the natural morphism
    $
            \mathbf{\Gamma}(\mathcal{O}_{K}, \;\cdot\;)
        \to
            \mathbf{\Gamma}(K, \;\cdot\;)
    $
\cite[Section 3.3]{Suz}.
Let $A, B$ be Abelian varieties dual to each other
and $\mathscr{A}, \mathscr{B}$ their N\'eron models.
Let $L / K$ be a finite Galois extension.
Then the morphism \eqref{0006} induces a morphism
	\[
				R \mathbf{\Gamma}_{x}(\mathcal{O}_{L}, \mathscr{A})
			\otimes^{L}
				R \mathbf{\Gamma}(\mathcal{O}_{L}, \mathscr{B}^{0})
		\to
			R \mathbf{\Gamma}_{x}(\mathcal{O}_{L}, \Gm[1])
		\cong
			\Z
	\]
by \cite[Proposition (3.3.4)]{Suz}
and hence a morphism
	\begin{equation} \label{0010}
				R \mathbf{\Gamma}_{x}(\mathcal{O}_{L}, \mathscr{A})
			\to
				R \mathbf{\Gamma}(\mathcal{O}_{L}, \mathscr{B}^{0})^{\mathrm{SD}}.
	\end{equation}
For $G \in \mathrm{IPAlg} / k$,
define
	$
			G^{\mathrm{SD}'}
		=
			\operatorname{\mathbf{Ext}}
			^{1}_{k^{\mathrm{indrat}}_{\mathrm{proet}}}
			(G, \Q / \Z)
	$.
Using the calculations presented in the following subsections, we shall show that $c(A)-c(B)$ vanishes by showing that this quantity equals its own additive inverse. In particular, sign and indexing conventions are essential.
\begin{proposition} \label{0014}
	Let $L / K$ be a finite Galois extension.
	Then the $n$-th cohomology objects $H^{n}$ of
	$R \mathbf{\Gamma}_{x}(\mathcal{O}_{L}, \mathscr{A})$
	for $n \in \Z$ are given by
		\begin{gather*}
					H^{1}
				\cong
					\frac{
						\mathbf{\Gamma}(L, A)
					}{
						\mathbf{\Gamma}(\mathcal{O}_{L}, \mathscr{A})
					},
			\quad
					H^{2}
				\cong
					\mathbf{H}^{1}(L, A),
			\\
					H^{n}
				=
					0,
				\quad
					n \ne 1, 2.
		\end{gather*}
	The $n$-th cohomology objects $H^{n}$ of
	$R \mathbf{\Gamma}(\mathcal{O}_{L}, \mathscr{B}^{0})^{\mathrm{SD}}$
	are given by
		\[
					H^{2}
				\cong
					\mathbf{\Gamma}(\mathcal{O}_{L}, \mathscr{B}^{0})^{\mathrm{SD}'},
			\quad
					H^{n}
				=
					0,
				\quad
					n \ne 2.
		\]
	Under these isomorphisms, the morphism \eqref{0010} in $H^{2}$
	is given by the composite
		\[
				\mathbf{H}^{1}(L, A)
			\cong
				\mathbf{\Gamma}(L, B)^{\mathrm{SD}'}
			\twoheadrightarrow
				\mathbf{\Gamma}(\mathcal{O}_{L}, \mathscr{B}^{0})^{\mathrm{SD}'}
		\]
	of the duality isomorphism (Proposition \ref{SuzDualityProp})
        and the natural morphism,
	the latter of which is surjective.
\end{proposition}

\begin{proof}
	The exact triangle
		\[
				R \mathbf{\Gamma}(\mathcal{O}_{L}, \mathscr{A})
			\to
				R \mathbf{\Gamma}(L, A)
			\to
				R \mathbf{\Gamma}_{x}(\mathcal{O}_{L}, \mathscr{A})[1]
		\]
	and $\mathbf{H}^{n}(\mathcal{O}_{L}, \mathscr{A}) = 0$ for $n \ge 1$ show
	the result for $R \mathbf{\Gamma}_{x}(\mathcal{O}_{L}, \mathscr{A})$.
	For $R \mathbf{\Gamma}(\mathcal{O}_{L}, \mathscr{B}^{0})^{\mathrm{SD}}$,
	we have
		\[
				R \mathbf{\Gamma}(\mathcal{O}_{L}, \mathscr{B}^{0})
			\cong
				\mathbf{\Gamma}(\mathcal{O}_{L}, \mathscr{B}^{0})
			\cong
				\mathbf{\Gamma}(\mathcal{O}_{L}, \mathscr{B})^{0}
		\]
	by \cite[Proposition Proposition (3.4.2) (b)]{Suz}, which is connected.
	Hence the result follows from \cite[Proposition (2.4.1) (a)]{Suz}.
	We have a commutative diagram
		\[
			\begin{CD}
				R \mathbf{\Gamma}(L, A)[-1]
			@>>>
				R \mathbf{\Gamma}(L, B)^{\mathrm{SD}}
			\\ @VVV @VVV \\
				R \mathbf{\Gamma}_{x}(\mathcal{O}_{L}, \mathscr{A})
			@>>>
				R \mathbf{\Gamma}(\mathcal{O}_{L}, \mathscr{B}^{0})^{\mathrm{SD}},
			\end{CD}
		\]
	which gives the description of the morphism on $H^{2}$.
	The surjectivity is the vanishing of
		$
			\operatorname{\mathbf{Ext}}
			^{\ge 2}_{k^{\mathrm{indrat}}_{\mathrm{proet}}}
			(\;\cdot\;, \Q / \Z)
		$
	for pro-algebraic groups \cite[Proposition (2.4.1) (a)]{Suz}.
\end{proof}

\begin{proposition} \label{0017}
	Let $A_{i}, B_{i}, \mathscr{A}_{i}, \mathscr{B}_{i}$ be
	as in Proposition \ref{0011}.
	Assume that the morphism $A_{1} \to A_{2}$ is an isogeny.
	Let $L / K$ be a finite Galois extension.
	Consider the morphism
		\begin{equation} \label{0018}
				R \mathbf{\Gamma}_{x}(
					\mathcal{O}_{L},
					[\mathscr{A}_{1} \to \mathscr{A}_{2}][-1]
				)
			\to
				R \mathbf{\Gamma}(
					\mathcal{O}_{L},
					[\mathscr{B}_{2}^{0} \to \mathscr{B}_{1}^{0}]
				)^{\mathrm{SD}}
		\end{equation}
	induced by the morphism \eqref{0012}.
	Then its mapping cone belongs to $\mathcal{D}$.
	The value of $\chi$ at this cone is equal to
	$c(A_{1})_{L}' - c(A_{2})_{L}' - c(B_{1})_{L}' + c(B_{2})_{L}'$.
\end{proposition}

\begin{proof}
	Denote the morphism in question as $C \to D$
	and a choice of its mapping cone as $E$.
	We have a morphism of exact triangles
		\[
			\begin{CD}
					C
				@>>>
					R \mathbf{\Gamma}_{x}(\mathcal{O}_{L}, \mathscr{A}_{1})
				@>>>
					R \mathbf{\Gamma}_{x}(\mathcal{O}_{L}, \mathscr{A}_{2})
				\\ @VVV @VVV @VVV \\
					D
				@>>>
					R \mathbf{\Gamma}(\mathcal{O}_{L}, \mathscr{B}_{1}^{0})^{\mathrm{SD}}
				@>>>
					R \mathbf{\Gamma}(\mathcal{O}_{L}, \mathscr{B}_{2}^{0})^{\mathrm{SD}}
			\end{CD}
		\]
	by \eqref{0013} and the naturality of the morphisms in 
        \cite[Proposition (3.3.4)]{Suz}.
	Under the isomorphisms in Proposition \ref{0014},
	the long exact sequence for the upper triangle can be written as
		\[
			\begin{CD}
					0
				@>>>
					H^{1}(C)
				@>>>
					\dfrac{
						\mathbf{\Gamma}(L, A_{1})
					}{
						\mathbf{\Gamma}(\mathcal{O}_{L}, \mathscr{A}_{1})
					}
				@>>>
					\dfrac{
						\mathbf{\Gamma}(L, A_{2})
					}{
						\mathbf{\Gamma}(\mathcal{O}_{L}, \mathscr{A}_{2})
					}
				\\
				@>>>
					H^{2}(C)
				@>>>
					\mathbf{H}^{1}(L, A_{1})
				@>>>
					\mathbf{H}^{1}(L, A_{2})
				@>>>
					0
			\end{CD}
		\]
	and $H^{n}(C) = 0$ for $n \ne 1, 2$,
	where the last surjectivity is the vanishing of
	$\mathbf{H}^{\ge 2}(L, \;\cdot\;)$ for finite flat group schemes
	\cite[Proposition (3.4.3) (b)]{Suz}.
	Let $H^{2}(C)'$ be the image of the morphism to $H^{2}(C)$ in this sequence
	and let $H^{2}(C)''$ be the image of the morphism from $H^{2}(C)$ in this sequence.
	The long exact sequence for the lower triangle reduces to an exact sequence
		\[
				0
			\to
				H^{2}(D)
			\to
				\mathbf{\Gamma}(\mathcal{O}_{L}, \mathscr{B}_{1}^{0})^{\mathrm{SD}'}
			\to
				\mathbf{\Gamma}(\mathcal{O}_{L}, \mathscr{B}_{2}^{0})^{\mathrm{SD}'}
			\to
				0
		\]
	and $H^{n}(D) = 0$ for $n \ne 2$,
	where the last morphism is surjective
	because the kernel of
		$
				\mathbf{\Gamma}(\mathcal{O}_{L}, \mathscr{B}_{2}^{0})
			\to
				\mathbf{\Gamma}(\mathcal{O}_{L}, \mathscr{B}_{1}^{0})
		$
	is contained in the kernel of
		$
				\mathbf{\Gamma}(L, B_{2})
			\to
				\mathbf{\Gamma}(L, B_{1})
		$
	(which is finite \'etale) and the vanishing of
		$
			\operatorname{\mathbf{Ext}}
			^{\ge 2}_{k^{\mathrm{indrat}}_{\mathrm{proet}}}
			(\;\cdot\;, \Q / \Z)
		$
	for pro-algebraic groups.
	Hence, with Proposition \ref{0014}, we have a commutative diagram with exact rows and columns
		\[
			\begin{CD}
				@.
				@.
					0
				@.
					0
				@.
				\\ @. @. @VVV @VVV @. \\
				@.
				@.
					\left( 
						\dfrac{
							\mathbf{\Gamma}(L, B_{1})
						}{
							\mathbf{\Gamma}(\mathcal{O}_{L}, \mathscr{B}_{1}^{0})
						}
					\right)^{\mathrm{SD}'}
				@>>>
					\left( 
						\dfrac{
							\mathbf{\Gamma}(L, B_{2})
						}{
							\mathbf{\Gamma}(\mathcal{O}_{L}, \mathscr{B}_{2}^{0})
						}
					\right)^{\mathrm{SD}'}
				@.
				\\ @. @. @VVV @VVV @. \\
					0
				@>>>
					H^{2}(C)''
				@>>>
					\mathbf{H}^{1}(L, A_{1})
				@>>>
					\mathbf{H}^{1}(L, A_{2})
				@>>>
					0
				\\ @. @VVV @VVV @VVV @. \\
					0
				@>>>
					H^{2}(D)
				@>>>
					\mathbf{\Gamma}(\mathcal{O}_{L}, \mathscr{B}_{1}^{0})^{\mathrm{SD}'}
				@>>>
					\mathbf{\Gamma}(\mathcal{O}_{L}, \mathscr{B}_{2}^{0})^{\mathrm{SD}'}
				@>>>
					0
				\\ @. @. @VVV @VVV @. \\
				@.
				@.
					0
				@.
					0.
				@.
			\end{CD}
		\]
	The morphism
		\[
				\dfrac{
					\mathbf{\Gamma}(L, B_{2})
				}{
					\mathbf{\Gamma}(\mathcal{O}_{L}, \mathscr{B}_{2}^{0})
				}
			\to
				\dfrac{
					\mathbf{\Gamma}(L, B_{1})
				}{
					\mathbf{\Gamma}(\mathcal{O}_{L}, \mathscr{B}_{1}^{0})
				}
		\]
	induces an isogeny on the semiabelian quotients
        by Proposition \ref{IsogL}.
	From these and Proposition \ref{0016}, we know that
	$H^{1}(C)$, $H^{2}(C)'$, $\Ker(H^{2}(C)'' \to H^{2}(D))^{0}$ and
	$\operatorname{coker}(H^{2}(C)'' \to H^{2}(D))^{0}$ are
	perfections of algebraic groups
	and
		\begin{gather*}
					\chi(H^{1}(C)) - \chi(H^{2}(C)')
				=
					c(A_{1})_{L}' - c(A_{2})_{L}',
			\\
					\chi(H^{2}(C)'' \to H^{2}(D))
				=
					- c(B_{1})_{L}' + c(B_{2})_{L}'.
		\end{gather*}
	Therefore $\Ker(H^{n}(C) \to H^{n}(D))$ and
	$\operatorname{coker}(H^{n}(C) \to H^{n}(D))$
	are objects of $\mathcal{D}$ for all $n$.
	Hence $E \in \mathcal{D}$.
	We have isomorphisms and an exact sequence
		\begin{gather*}
					H^{0}(E)
				\cong
					H^{1}(C),
			\\
					0
				\to
					H^{2}(C)'
				\to
					H^{1}(E)
				\to
					\Ker(H^{2}(C)'' \to H^{2}(D))
				\to
					0,
			\\
					H^{2}(E)
				\cong
					\operatorname{coker}(H^{2}(C)'' \to H^{2}(D)),
			\\
					H^{n}(E)
				=
					0,
			\quad
					n \ne 0, 1, 2.
		\end{gather*}
	Applying $\chi$, we get the result.
\end{proof}


\subsection{Invariance of isogeny conductors under duality}

Assume that $K$ has equal characteristic.
Let $A_{1} \to A_{2}$ be an isogeny of Abelian varieties over $K.$ 
Let $B_{2} \to B_{1}$ be the dual isogeny.
Let $\mathscr{A}_{1} \to \mathscr{A}_{2}$
and $\mathscr{B}_{2} \to \mathscr{B}_{1}$ be
the morphisms induced on the N\'eron models. Let $L / K$ be a finite Galois extension.
We shall show below that the equality $$c(A_{1})_{L}' - c(A_{2})_{L}' = c(B_{1})_{L}' - c(B_{2})_{L}'$$
holds. Choosing $L$ such that one (hence all) of the semiabelian varieties just mentioned have semiabelian reduction over $\Og_L$, dividing by $e_{L/K}$ and using Proposition \ref{0015} shows $c(A_1)-c(A_2)=c(B_1)-c(B_2).$ In particular, if $A_1$ and $A_2$ are dual to one another, this implies $c(A_1)=c(A_2).$ The kernel $F$ of the isogeny $A_1 \to A_2$ admits a canonical filtration 
$$0=F_0 \subset F_1 \subset F_2 \subset F_3 \subset F_4=F$$
whose associated graded object is the direct sum of an infinitesimal multiplicative group scheme, an infinitesimal unipotent group scheme, a finite étale group scheme of $p$-power order, and a finite étale group scheme of order not divisible by $p,$ respectively. This filtration induces a factorisation of the morphisms $A_1\to A_2$ and $B_2 \to B_1,$ so we may assume without loss of generality that $F$ itself belongs to one of the classes of finite $K$-group schemes just listed. Note that, if $F$ is étale of order invertible in $\Og_K,$ then the maps $\Lie \mathscr{A}_1 \to \Lie \mathscr{A}_2$ and $\Lie \mathscr{B}_2 \to \Lie \mathscr{B}_1$ are isomorphisms. In particular, we may assume that the isogeny $A_1 \to A_2$ has $p$-power degree.

\begin{proposition} \label{0019}
    Assume that $A_{1} \to A_{2}$ has
    multiplicative (hence connected) kernel.
    Then $c(A_{1})_{L}' - c(A_{2})_{L}' = c(B_{1})_{L}' - c(B_{2})_{L}'$.
\end{proposition}

\begin{proof}
	The morphism
		$
				\mathbf{\Gamma}(\mathcal{O}_{L}, \mathscr{B}_{2}^{0})
			\to
				\mathbf{\Gamma}(\mathcal{O}_{L}, \mathscr{B}_{1}^{0})
		$
	has finite \'etale kernel
	and induces an isogeny on the semiabelian quotients.
	For each $i$,
	the group $\mathbf{\Gamma}(\mathcal{O}_{L}, \mathscr{B}_{i}^{0})$
	is the perfection of the Greenberg transform of infinite level of $\mathscr{B}_{i}^{0}$.
	By assumption, the generic fiber of the morphism
	$\mathscr{B}_{2}^{0} \to \mathscr{B}_{1}^{0}$
	is \'etale surjective.
	Therefore the cokernel of
		$
				\mathbf{\Gamma}(\mathcal{O}_{L}, \mathscr{B}_{2}^{0})
			\to
				\mathbf{\Gamma}(\mathcal{O}_{L}, \mathscr{B}_{1}^{0})
		$
	is the perfection of a unipotent algebraic group
	whose dimension is the $\mathcal{O}_{L}$-length of
	the cokernel of
		\[
				\Lie(\mathscr{B}_{2}^{0}) \otimes_{\mathcal{O}_{K}} \mathcal{O}_{L}
			\to
				\Lie(\mathscr{B}_{1}^{0}) \otimes_{\mathcal{O}_{K}} \mathcal{O}_{L}
		\]
	by \cite[Theorem 2.1 (b)]{LLR}.
	This $\mathcal{O}_{L}$-length divided by $e_{L / K}$ is independent of the choice of $L$.
	With Proposition \ref{0016},
	we know that
		$
			R \mathbf{\Gamma}(
				\mathcal{O}_{L},
				[\mathscr{B}_{2}^{0} \to \mathscr{B}_{1}^{0}]
			)^{\mathrm{SD}}
		$
	is an object of $\mathcal{D}_{u}$
	and the value of $\chi / e_{L / K}$ at this object is independent of the choice of $L$.
	
	For
		$
			R \mathbf{\Gamma}_{x}(
				\mathcal{O}_{L},
				[\mathscr{A}_{1} \to \mathscr{A}_{2}][-1]
			)
		$,
	let $N$ be the kernel of $A_{1} \to A_{2}$.
	Let $\mathscr{N}$ be the schematic closure of $N$ in $\mathscr{A}_{1}$,
	which is finite flat over $\mathcal{O}_{K}$
        since $N$ is infinitesimal.
	Let $\mathscr{A}_{2}'$ be the fppf quotient $\mathscr{A}_{1} / \mathscr{N}$,
	which is a smooth separated group scheme over $\mathcal{O}_{K}$
        by \cite[Théorème 4.C]{An}.
        The term-wise exact sequence of complexes
            \[
                    0
                \to
                    \mathscr{N}
                \to
                    [\mathscr{A}_{1} \to \mathscr{A}_{2}][-1]
                \to
                    [\mathscr{A}_{2}' \to \mathscr{A}_{2}][-1]
                \to
                    0
            \]
        induces an exact triangle
		\[
				R \mathbf{\Gamma}_{x}(\mathcal{O}_{L}, \mathscr{N})
			\to
				R \mathbf{\Gamma}_{x}(
					\mathcal{O}_{L},
					[\mathscr{A}_{1} \to \mathscr{A}_{2}][-1]
				)
			\to
				R \mathbf{\Gamma}_{x}(
					\mathcal{O}_{L},
					[\mathscr{A}_{2}' \to \mathscr{A}_{2}][-1]
				).
		\]
	    Let $\mathscr{M}$ be the Cartier dual of $\mathscr{N}$.
	We have a canonical isomorphism
		\[
				R \mathbf{\Gamma}_{x}(\mathcal{O}_{L}, \mathscr{N})
			\cong
				R \mathbf{\Gamma}(\mathcal{O}_{L}, \mathscr{M})^{\mathrm{SD}}[-1]
		\]
	by \cite[Theorem (5.2.1.2)]{Suz}.
        Proposition \ref{EtCriterion} shows that these isomorphic objects are in $\mathcal{D}$ and that the
	dimension of $\mathbf{H}^{1}(\mathcal{O}_{L}, \mathscr{M})$
	is the $\mathcal{O}_{L}$-length of the pullback of
	$\Omega^{1}_{\mathscr{M} \times_{\mathcal{O}_{K}} \mathcal{O}_{L} / \mathcal{O}_{L}}$
	along the zero section.
	Hence the values of $\chi / e_{L / K}$ at
	$R \mathbf{\Gamma}(\mathcal{O}_{L}, \mathscr{M})^{\mathrm{SD}}$
	and hence at $R \mathbf{\Gamma}_{x}(\mathcal{O}_{L}, \mathscr{N})$
	are independent of the choice of $L$.
	
	The morphism $\mathscr{A}_{2}' \to \mathscr{A}_{2}$ is an isomorphism over $K$.
	Hence
		\[
				R \mathbf{\Gamma}_{x}(
					\mathcal{O}_{L},
					[\mathscr{A}_{2}' \to \mathscr{A}_{2}][-1]
				)
			\cong
				R \mathbf{\Gamma}(
					\mathcal{O}_{L},
					[\mathscr{A}_{2}' \to \mathscr{A}_{2}][-1]
				).
		\]
	By the same reasoning as the case of
    		$
    			R \mathbf{\Gamma}(
    				\mathcal{O}_{L},
    				[\mathscr{B}_{2}^{0} \to \mathscr{B}_{1}^{0}]
    			)
    		$
            above,
        this object is in $\mathcal{D}$ and
	the value of $\chi / e_{L / K}$ at this object is independent of the choice of $L$.
	
	With Proposition \ref{0017},
	all these together imply that
	$c(A_{1})_{L}' - c(A_{2})_{L}' - c(B_{1})_{L}' + c(B_{2})_{L}'$
	divided by $e_{L / K}$ is independent of the choice of $L$.
	Hence we may assume that $L = K$.
	But then every term becomes zero.
\end{proof}

\begin{proposition} \label{0020}
	Assume that $A_{1} \to A_{2}$ has \'etale kernel.
	Then $c(A_{1})_{L}' - c(A_{2})_{L}' = c(B_{1})_{L}' - c(B_{2})_{L}'$.
\end{proposition}

\begin{proof}
	The dual isogeny $B_{2} \to B_{1}$ has multiplicative kernel.
	Hence the proposition follows from Proposition \ref{0019}
	by switching $A_{i}$ and $B_{i}$.
\end{proof}

\begin{proposition} \label{0021}
	Assume that $A_{1} \to A_{2}$ has unipotent connected kernel.
	Then $c(A_{1})_{L}' - c(A_{2})_{L}' = c(B_{1})_{L}' - c(B_{2})_{L}'$.
\end{proposition}

\begin{proof}
	Let $N$ and $M$ be the kernels of $A_{1} \to A_{2}$ and $B_{2} \to B_{1}$, respectively.
	Let $\mathscr{N}$ and $\mathscr{M}$ be the schematic closures of
	$N$ in $\mathscr{A}_{1}$ and $M$ in $\mathscr{B}_{2}$, respectively.
	Then both $\mathscr{N}$ and $\mathscr{M}$ are finite flat.
	Let $\mathscr{A}_{2}'$ and $\mathscr{B}_{1}'$ be the fppf quotients
	$\mathscr{A}_{1} / \mathscr{N}$ and $\mathscr{B}_{2} / \mathscr{M}$, respectively,
	which are smooth separated group schemes.
	Let $\mathscr{N}'$ be the Cartier dual of $\mathscr{M}$
	and $\mathscr{M}'$ be the Cartier dual of $\mathscr{N}$.
    We have $\mathscr{M} \subset \mathscr{B}_{2}^{0}$
        and $\mathscr{B}_{2}^{0} / \mathscr{M} \cong \mathscr{B}_{1}'^{0}$.
	By \eqref{0012}, we have natural morphisms
		\begin{equation} \label{0029}
				\mathscr{N} \otimes^{L} \mathscr{M}
			\to
					[\mathscr{A}_{1} \to \mathscr{A}_{2}][-1]
				\otimes^{L}
					[\mathscr{B}_{2}^{0} \to \mathscr{B}_{1}^{0}][-1]
			\to
				\Gm
		\end{equation}
	in $D(\mathcal{O}_{K, \fppf})$.
	Over the generic fibers, it gives the Cartier duality between $N$ and $M$.
	This fits in the commutative diagram
		\begin{equation} \label{0028}
			\begin{CD}
					R \mathbf{\Gamma}_{x}(\mathcal{O}_{L}, \mathscr{N})
				@>>>
					R \mathbf{\Gamma}(\mathcal{O}_{L}, \mathscr{M})^{\mathrm{SD}}[-1]
				\\ @VVV @AAA \\
					R \mathbf{\Gamma}_{x}(
						\mathcal{O}_{L},
						[\mathscr{A}_{1} \to \mathscr{A}_{2}][-1]
					)
				@>>>
					R \mathbf{\Gamma}(
						\mathcal{O}_{L},
						[\mathscr{B}_{2}^{0} \to \mathscr{B}_{1}^{0}]
					)^{\mathrm{SD}}
			\end{CD}
		\end{equation}
	along with the morphism \eqref{0018}.
	As in the final step in the proof of Proposition \ref{0019},
	it is enough to show that
	the value of $\chi / e_{L / K}$ at a mapping cone
	of the lower horizontal morphism
	is independent of the choice of $L / K$.
	For this, by the octahedral axiom,
	it is enough to show that
	the value of $\chi / e_{L / K}$ at a mapping cone
	of any other three morphisms in the diagram
	is well-defined and independent of the choice of $L / K$.
	
	For the right vertical morphism of \eqref{0028},
	we have an exact triangle
		\[
				R \mathbf{\Gamma}(\mathcal{O}_{L}, \mathscr{M})
			\to
				R \mathbf{\Gamma}(
					\mathcal{O}_{L},
					[\mathscr{B}_{2}^{0} \to \mathscr{B}_{1}^{0}]
				)[-1]
			\to
				R \mathbf{\Gamma}(
					\mathcal{O}_{L},
					[\mathscr{B}_{1}'^{0} \to \mathscr{B}_{1}^{0}]
				)[-1].
		\]
	Since $\mathscr{B}_{1}'^{0} \to \mathscr{B}_{1}^{0}$ is an isomorphism
	on the generic fibers,
	the value of $\chi / e_{L / K}$ at the third term is
	well-defined and independent of the choice of $L / K$
	by the same argument as the first paragraph of the proof of
	Proposition \ref{0019}.
	By Proposition \ref{0016},
	we get the desired statement in this case.
	
	For the left vertical morphism of \eqref{0028},
	we have an exact triangle
		\[
				R \mathbf{\Gamma}_{x}(\mathcal{O}_{L}, \mathscr{N})
			\to
				R \mathbf{\Gamma}_{x}(
					\mathcal{O}_{L},
					[\mathscr{A}_{1} \to \mathscr{A}_{2}]
				)[-1]
			\to
				R \mathbf{\Gamma}_{x}(
					\mathcal{O}_{L},
					[\mathscr{A}_{2}' \to \mathscr{A}_{2}]
				)[-1].
		\]
	Since $\mathscr{A}_{2}' \to \mathscr{A}_{2}$ is an isomorphism
	on the generic fibers, we have
		\[
				R \mathbf{\Gamma}_{x}(
					\mathcal{O}_{L},
					[\mathscr{A}_{2}' \to \mathscr{A}_{2}]
				)
			\cong
				R \mathbf{\Gamma}(
					\mathcal{O}_{L},
					[\mathscr{A}_{2}' \to \mathscr{A}_{2}]
				).
		\]
	The same argument as the previous case suffices.
	
	For the upper horizontal morphism of \eqref{0028},
	let $\mathscr{N} \to \mathscr{N}'$ be the morphism induced by
	\eqref{0029}.
	Then the upper horizontal morphism of \eqref{0028}
	can be written as the induced morphism
		\[
				R \mathbf{\Gamma}_{x}(\mathcal{O}_{L}, \mathscr{N})
			\to
				R \mathbf{\Gamma}_{x}(\mathcal{O}_{L}, \mathscr{N}')
		\]
	by \cite[Theorem (5.2.1.2)]{Suz}.
	We have
		\[
				R \mathbf{\Gamma}_{x}(\mathcal{O}_{L}, [\mathscr{N} \to \mathscr{N}'])
			\cong
				R \mathbf{\Gamma}(\mathcal{O}_{L}, [\mathscr{N} \to \mathscr{N}'])
		\]
	since $\mathscr{N} \to \mathscr{N}'$ is an isomorphism
	on the generic fibers.
	Note that these generic fibers are not necessarily smooth.
	Let $G$ be the Weil restriction $\Res_{\mathscr{M}' / \mathcal{O}_{K}} \Gm$,
	which is a smooth affine group scheme over $\mathcal{O}_{K}$.
	Let $\mathscr{N} \hookrightarrow G$ be the natural inclusion.
	Set $H = G / \mathscr{N}$,
	which is a smooth affine group scheme over $\mathcal{O}_{K}$.
	Let $G'$ be the Weil restriction $\Res_{\mathscr{M} / \mathcal{O}_{K}} \Gm$.
	Let $\mathscr{N}' \hookrightarrow G'$ be the natural inclusion
	and set $H' = G' / \mathscr{N}'$.
	We have a commutative diagram with exact rows
		\[
			\begin{CD}
					0
				@>>>
					\mathscr{N}
				@>>>
					G
				@>>>
					H
				@>>>
					0,
				\\ @. @VVV @VVV @VVV @. \\
					0
				@>>>
					\mathscr{N}'
				@>>>
					G'
				@>>>
					H'
				@>>>
					0.
			\end{CD}
		\]
	(These are B\'egueri's canonical smooth resolutions;
	see \cite[Chapter III, Theorem A.5]{Milne}.)
	The vertical morphisms are isomorphisms on the generic fibers.
	The diagram induces an exact triangle
		\[
				R \mathbf{\Gamma}(\mathcal{O}_{L}, [\mathscr{N} \to \mathscr{N}'])
			\to
				\frac{
					\mathbf{\Gamma}(\mathcal{O}_{L}, G')
				}{
					\mathbf{\Gamma}(\mathcal{O}_{L}, G)
				}
			\to
				\frac{
					\mathbf{\Gamma}(\mathcal{O}_{L}, H')
				}{
					\mathbf{\Gamma}(\mathcal{O}_{L}, H)
				}.
		\]
        As before, the values of $\chi / e_{L / K}$ at the second and third terms
	are well-defined and independent of the choice of $L$.
	Hence the same is true at the first term.
\end{proof}

\begin{proposition} \label{0022}
	We have $c(A_{1})_{L}' - c(A_{2})_{L}' = c(B_{1})_{L}' - c(B_{2})_{L}'$.
	In particular, we have
	$c(A_{1}) - c(A_{2}) = c(B_{1}) - c(B_{2})$.
\end{proposition}

\begin{proof}
	This follows from Propositions \ref{0017}, \ref{0019}, \ref{0020} and \ref{0021}.
\end{proof}

\begin{theorem} \label{0023}
	The base change conductor is invariant under duality for Abelian varieties over $K$.
	More precisely, we have $c(A) = c(B)$
	for any Abelian varieties $A$ and $B$ over $K$ dual to each other.
\end{theorem}

\begin{proof}
	Apply Proposition \ref{0022} to any isogeny $A \to B$ to the dual.
	We obtain $c(A) - c(B) = c(B) - c(A)$.
	Thus $c(A) = c(B)$.
\end{proof}

\begin{remark} \label{0030}
	The proof of Proposition \ref{0019} also works for mixed characteristic $K$
	with little modifications.
	In this case, $\mathcal{N}$ is only quasi-finite flat separated
	and not necessarily finite.
	We have
		$
				R \mathbf{\Gamma}(\mathcal{O}_{L}, \mathcal{N})
			\cong
				R \mathbf{\Gamma}(\mathcal{O}_{L}, \mathcal{N}^{\mathrm{f}})
		$
	by \cite[Proposition (5.2.3.5)]{Suz},
	where $\mathcal{N}^{\mathrm{f}}$ is the finite part of $\mathcal{N}$.
	We additionally need to know that
	$R \mathbf{\Gamma}(L, N)$ for mixed characteristic $K$
	is an object of $\mathcal{D}_{u}$
	and the value of $\chi / e_{L / K}$ at this object is
	independent of the choice of $L$.
	But these follow from Proposition \ref{SuzFinFlatProp}
	and \cite[Proposition 4.3.2, Theorem 4.3.3]{Beg},
	where the dimension of $\mathbf{H}^{1}(L, N)$ is
	shown to be $v_{L}(\# N)$
	(where $v_{L}$ is the normalized valuation for $L$
	and $\# N$ the order of $N$).
\end{remark}


\subsection{Isogeny invariance for tori}

Assume that $K$ has equal characteristic.

\begin{theorem}[{cf.\ \cite{CY}}] \label{0025}
	The base change conductor for tori is isogeny invariant.
	More precisely, if $T_{1} \to T_{2}$ is an isogeny between tori over $K$,
	then $c(T_{1}) = c(T_{2})$.
\end{theorem}

\begin{proof}
	Let $N$ be the kernel of $T_{1} \to T_{2}$.
	Let $\mathscr{T}_{i}$ be the N\'eron lft-model of $T_{i}$.
	Let $L / K$ be any finite Galois extension. Note first that the maximal étale quotient of $N$ has order invertible in $\Og_K$ because $N$ is multiplicative. In particular, we can factorise $T_1\to T_2$ into an isogeny with infinitesimal kernel and an isogeny of degree prime to $p.$ It clearly suffices to show invariance of $c(-)$ for the two isogenies separately. If the degree is invertible in $\Og_K,$ then the map $\Lie \mathscr{T}_1 \to \Lie \mathscr{T}_2$ is an isomorphism, and the same holds true after base change to $L$ (with Néron lft-models taken over $\Og_L$). Therefore we shall henceforth assume that $N$ is infinitesimal.
    
	We have $\mathbf{H}^{n}(L, T_{i}) = 0$ for all $n \ge 1$
	by \cite[Proposition (3.4.3) (e)]{Suz}.
	Also we have $\mathbf{\Gamma}(L, N) = 0$ since $N$ is infinitesimal.
	Therefore we have a commutative diagram with exact rows and columns
		\[
			\begin{CD}
				@.
					0
				@.
					0
				\\
				@. @VVV @VVV
				\\
					0
				@>>>
					\mathbf{\Gamma}(\mathcal{O}_{L}, \mathscr{T}_{1})
				@>>>
					\mathbf{\Gamma}(\mathcal{O}_{L}, \mathscr{T}_{2})
				@>>>
					\dfrac{
						\mathbf{\Gamma}(\mathcal{O}_{L}, \mathscr{T}_{2})
					}{
						\mathbf{\Gamma}(\mathcal{O}_{L}, \mathscr{T}_{1})
					}
				@>>>
					0
				\\
				@. @VVV @VVV @VVV
				\\
					0
				@>>>
					\mathbf{\Gamma}(L, T_{1})
				@>>>
					\mathbf{\Gamma}(L, T_{2})
				@>>>
					\mathbf{H}^{1}(L, N)
				@>>>
					0
				\\
				@. @VVV @VVV
				\\
				@.
					\dfrac{
						\mathbf{\Gamma}(L, T_{1})
					}{
						\mathbf{\Gamma}(\mathcal{O}_{L}, \mathscr{T}_{1})
					}
				@>>>
					\dfrac{
						\mathbf{\Gamma}(L, T_{2})
					}{
						\mathbf{\Gamma}(\mathcal{O}_{L}, \mathscr{T}_{2})
					}
				\\
				@. @VVV @VVV
				\\
				@.
					0
				@.
					0.
			\end{CD}
		\]
        Therefore the morphism
            \begin{equation} \label{0027}
					\frac{
						\mathbf{\Gamma}(L, T_{1})
					}{
						\mathbf{\Gamma}(\mathcal{O}_{L}, \mathscr{T}_{1})
					}
				\to
					\frac{
						\mathbf{\Gamma}(L, T_{2})
					}{
						\mathbf{\Gamma}(\mathcal{O}_{L}, \mathscr{T}_{2})
					}
            \end{equation}
        and the morphism
		\begin{equation} \label{0024}
				\frac{
					\mathbf{\Gamma}(\mathcal{O}_{L}, \mathscr{T}_{2})
				}{
					\mathbf{\Gamma}(\mathcal{O}_{L}, \mathscr{T}_{1})
				}
			\to
				\mathbf{H}^{1}(L, N)
		\end{equation}
        have isomorphic kernels and isomorphic cokernels.
	By Proposition \ref{IsogL}, the morphism \eqref{0027}
	has mapping cone in $\mathcal{D}_{u}$
	and the value of $\chi$ at this cone is $- c(T_{1})_{L}' + c(T_{2})_{L}'$.
        Hence the same properties hold for the morphism \eqref{0024}.
	Therefore it is enough to show that
	the value of $\chi / e_{L / K}$ at the cone of the morphism \eqref{0024} is
        independent of the choice of $L$.
	
	Let $\mathscr{N}$ be the schematic closure of $N$ in $\mathscr{T}_{1}$,
	which is finite flat over $\mathcal{O}_{K}$.
	Let $\mathscr{T}_{2}'$ be the fppf quotient $\mathscr{T}_{1} / \mathscr{N}$,
	which is a smooth separated group scheme over $\mathcal{O}_{K}$.
	Consider the induced two morphisms
		\[
				\mathbf{\Gamma}(\mathcal{O}_{L}, \mathscr{T}_{1})
			\to
				\mathbf{\Gamma}(\mathcal{O}_{L}, \mathscr{T}_{2}')
			\to
				\mathbf{\Gamma}(\mathcal{O}_{L}, \mathscr{T}_{2}).
		\]
	The cokernel of the first morphism is
	$\mathbf{H}^{1}(\mathcal{O}_{L}, \mathscr{N})$
	since $\mathbf{H}^{1}(\mathcal{O}_{L}, \mathscr{T}_{1}) = 0$
	by \cite[Proposition (3.4.2) (a)]{Suz}.
	The kernel of the second morphism is zero
	since $\mathscr{T}_{2}' \to \mathscr{T}_{2}$ is an isomorphism over $L$.
	Therefore we have an exact sequence
		\[
				0
			\to
				\mathbf{H}^{1}(\mathcal{O}_{L}, \mathscr{N})
			\to
				\frac{
					\mathbf{\Gamma}(\mathcal{O}_{L}, \mathscr{T}_{2})
				}{
					\mathbf{\Gamma}(\mathcal{O}_{L}, \mathscr{T}_{1})
				}
			\to
				\frac{
					\mathbf{\Gamma}(\mathcal{O}_{L}, \mathscr{T}_{2})
				}{
					\mathbf{\Gamma}(\mathcal{O}_{L}, \mathscr{T}_{2}')
				}
			\to
				0.
		\]
	Consider the resulting morphisms
		\[
				\mathbf{H}^{1}(\mathcal{O}_{L}, \mathscr{N})
			\hookrightarrow
				\frac{
					\mathbf{\Gamma}(\mathcal{O}_{L}, \mathscr{T}_{2})
				}{
					\mathbf{\Gamma}(\mathcal{O}_{L}, \mathscr{T}_{1})
				}
			\to
				\mathbf{H}^{1}(L, N).
		\]
	Their composite is injective by \cite[Proposition (3.4.6)]{Suz}.
	Hence the value of $\chi$ at the morphism \eqref{0024} is equal to
		\[
				- \dim
				\frac{
					\mathbf{\Gamma}(\mathcal{O}_{L}, \mathscr{T}_{2})
				}{
					\mathbf{\Gamma}(\mathcal{O}_{L}, \mathscr{T}_{2}')
				}
			+
				\dim
				\frac{
					\mathbf{H}^{1}(L, N)
				}{
					\mathbf{H}^{1}(\mathcal{O}_{L}, \mathscr{N})
				}.
		\]
	The first dimension divided by $e_{L / K}$ is independent of $L$
	by \cite[Theorem 2.1 (b)]{LLR} as before.
	The second dimension divided by $e_{L / K}$ is independent of $L$
	by Propositions \ref{FinDuality} and \ref{EtCriterion}.
	Combining these two, we get the result.
\end{proof}

Again, a slight modification of this proof also works in the mixed characteristic case (see Remark \ref{0030}).


\textsc{Mathematisches Institut der Heinrich-Heine-Universität Düsseldorf, Universitätsstr. 1, 40225 Düsseldorf, Germany} \\
\it E-mail address: \rm \texttt{otto.overkamp@uni-duesseldorf.de}\\
\\
\textsc{
   Department of Mathematics, Chuo University,
   1-13-27 Kasuga, Bunkyo-ku, Tokyo 112-8551, Japan
} \\
\it E-mail address: \rm \texttt{tsuzuki@gug.math.chuo-u.ac.jp} 
\end{document}